\documentclass[reqno, final]{amsart}
\usepackage{rmwe}  
\usepackage[parfill]{parskip}
\usepackage{enumitem}
\usepackage{dsfont}
\usepackage{mathrsfs}
\usepackage{tikzit}

\tikzstyle{black vertex}=[fill=black, draw=none, shape=circle, tikzit category=nodes]
\tikzstyle{blue vertex}=[fill=blue, draw=none, shape=circle, tikzit category=nodes]
\tikzstyle{red vertex}=[fill=red, draw=none, shape=circle, tikzit category=nodes]
\tikzstyle{G{n+1}}=[fill=red, draw=black, shape=circle, tikzit category=vertex set colour scheme]
\tikzstyle{smooth n}=[fill=blue, draw=black, shape=circle, tikzit category=vertex set colour scheme]
\tikzstyle{G{n}}=[fill={rgb,255: red,0; green,150; blue,0}, draw=black, shape=circle, tikzit category=vertex set colour scheme]
\tikzstyle{empty}=[fill=none, draw=black, shape=circle, tikzit category=vertex set colour scheme, tikzit fill=white]
\tikzstyle{G-bdry}=[fill=blue, draw=gamdomcol, shape=circle, thick, tikzit category=vertex set colour scheme]
\tikzstyle{white}=[fill=white, draw=black, shape=circle]
\tikzstyle{G{n+1}-bdry}=[fill={rgb,255: red,128; green,0; blue,128}, draw=gamdomcol, shape=circle, tikzit category=vertex set colour scheme, thick]
\tikzstyle{1/2 smooth n}=[fill=blue, draw=black, shape=circle, tikzit category=vertex set colour scheme, scale={1/2}]
\tikzstyle{1/3 smooth n}=[fill=blue, draw=black, shape=circle, tikzit category=vertex set colour scheme, scale={1/3}]
\tikzstyle{1/2G{n}q}=[fill={rgb,255: red,0; green,150; blue,0}, draw=black, shape=circle, scale={1/2}]
\tikzstyle{1/4 smooth n}=[fill=blue, draw=black, shape=circle, tikzit category=vertex set colour scheme, scale={1/4}]
\tikzstyle{1/2}=[fill=none, draw=black, shape=circle, tikzit category=vertex set colour scheme, scale=0.5, tikzit fill=white]

\tikzstyle{v thick dots}=[-, mydots]
\tikzstyle{mydashed}=[-, dashed]
\tikzstyle{arrow}=[->]
\tikzstyle{grey edge}=[-, draw={rgb,255: red,180; green,180; blue,180}]
\tikzstyle{G-domain}=[-, draw=gamdomcol]
\tikzstyle{thick arrow}=[->, draw=black, very thick]
\tikzstyle{mapst}=[{|->}]
\tikzstyle{mythick}=[-, thick]
\tikzstyle{thickdashed}=[-, dashed, thick]
\tikzstyle{dashed+G-domain}=[-, draw=gamdomcol, dashed]

\makeatletter
\tikzset{
    dot diameter/.store in=\dot@diameter,
    dot diameter=3pt,
    dot spacing/.store in=\dot@spacing,
    dot spacing=5mm,
    mydots/.style={
        line width=\dot@diameter,
        line cap=round,
        dash pattern=on 0pt off \dot@spacing
    }
}
\makeatother

\newcommand{\gampluscol}{\color{red}}
\newcommand{\gamcol}{\color{rgb, 255: red, 0; green, 150; blue, 0}}
\newcommand{\smgamcol}{\color{blue}}
\definecolor{gamdomcol}{RGB}{135, 180, 0}
\newcommand{\gamdomcol}{\color{gamdomcol}}

\setlist[enumerate]{before=\setlength{\baselineskip}{20pt}, itemsep=0pt}
\setlist[itemize]{before=\setlength{\baselineskip}{20pt}, itemsep=0pt}

\newtheorem{clm}{Claim}

\newcommand{\E}{\mathcal E}
\newcommand{\F}{\mathcal F}
\newcommand{\J}{\mathcal J}
\newcommand{\V}{\mathcal V_x}
\newcommand{\T}{\mathcal T}

\newcommand{\K}{\mathbb K}
\newcommand{\hk}{\mathbb{\hat K}}
\renewcommand{\k}{k}

\newcommand{\bH}{\mathbb{H}}
\newcommand{\CD}{\overline D}

\newcommand{\an}{\text{an}}

\newcommand{\mg}{\mathfrak g}
\newcommand{\mm}{\mathfrak m}
\newcommand{\q}{q}
\newcommand{\bvec}[1]{{\bf#1}}
\newcommand{\emp}{\emptyset}

\newcommand{\nin}{\notin}

\newcommand{\dilating}{dilating{}}

\renewcommand{\and}{\text{ and }}
\newcommand{\dashto}{\dashrightarrow}
\newcommand{\GV}{\mathcal{X}}
\newcommand{\Dir}[1]{T_{#1}\P^1_\an}

\newcommand{\numericallyrepelling}{numerically repelling}

\newcommand\capcirc{\mathbin{\ooalign{$\cap$\cr
  \hidewidth\raise.2ex\hbox{\scalebox{0.4}{\mbox{$\circ$}}}\hidewidth\cr}}}

\DeclareMathOperator{\PGL}{PGL}

\DeclareMathOperator{\rdeg}{rdeg}

\DeclareMathOperator{\Ram}{Ram}
\DeclareMathOperator{\Crit}{Crit}
\DeclareMathOperator{\Hull}{Hull}
\DeclareMathOperator{\Inj}{Inj}
\DeclareMathOperator{\GCD}{GCD}

\DeclareMathOperator{\Frac}{Frac}

\DeclareMathOperator{\cupdot}{\mathbin{\mathaccent\cdot\cup}}
\DeclareMathOperator{\redct}{red}

\DeclareMathOperator{\Gal}{Gal}

\newcommand{\dmult}{multiplier}

\makeatletter
\newcommand{\abs}[2][\@nil]{%
  \def\tmp{#1}%
   \ifx\tmp\@nnil
       \left\lvert#2\right\rvert
    \else
         \left\lvert#2\right\rvert_{#1}
    \fi}
\makeatother

\newtheorem*{thm*}{Theorem}

\def\thesisbool{0}
\newcommand{\thesisarticle}[2]{\if\thesisbool1{#1}\else{#2}\fi}

\begin{document}
  \title{Algebraic Stability for Skew Products}
 \author{\mylongname}
 
\begin{abstract}
In this article we study algebraic stability for rational skew products in two dimensions $\phi : X \dashto X$, i.e. maps of the form $\phi(x, y) = (\phi_1(x), \phi_2(x, y))$.  
We prove that when $X$ is a birationally ruled surface and $\phi_1$ has no superattracting cycles, then we can always find a \emph{smooth} surface $\hat X$ and an algebraic stabilisation $\pi : (\hat \phi, \hat X) \to (\phi, X)$ which is a \emph{birational morphism}. We provide an example of a skew product $\phi$ where $\phi_1$ has a superattracting fixed point and $\phi$ is not algebraically stable on any model.

Our techniques involve transforming the stabilisation issue into a combinatorial dynamical problem for a `non-Archimedean skew product' $\phi_*: \P^1_\an(\K) \to \P^1_\an(\K)$ on the Berkovich projective line over the Puiseux series, $\K$. The Fatou-Julia theory for $\phi_*$ is instrumental to our approach.
\end{abstract}
 
 \maketitle
 
 \tableofcontents
 
 \clearpage


\section{Introduction} 

The dynamics of rational maps $\phi : X \dashto X$ are often complicated by their lack of continuity. There is always a pullback (or pushforward) action by $\phi$ on the divisors of $X$, but this is not necessarily compatible with iteration. 
We say a rational map $\phi$ is \emph{algebraically stable} iff 
\begin{equation}\label{eqn:AS}
 \forall n \in \N\quad (\phi^*)^n = (\phi^n)^*
\end{equation}
\cite{FS2, Sib}.
This property is more reasonable to hope for than continuity and important for finding dynamical invariants. For instance, the calculation of the dynamical degree, $\lambda_1(\phi) = \lim_{n\to\infty} \|(\phi^n)^*\|^{\frac 1n}$, of an algebraically stable mapping reduces through linear algebra to finding the spectral radius of $f^*$ on $H^{1, 1}(X)$. 
Further, in order to construct invariant measures or currents, it is often necessary for $\phi$ to be algebraically stable to control its dynamical indeterminacy; see \cite{BD, Gue, Gue05, DDG1, DDG2, DDG3}.

On a smooth surface, algebraic stability admits a geometric characterisation that is more accessible in practice; see \cite{roeder}. A \emph{destabilising orbit} is an orbit of (closed) points $p, \phi(p)\dots, \phi^{n-1}(p)$ in $X$, where $\phi^{n-1}(p)$ is an indeterminate point for $\phi$ and $\phi^{-1}(p)$ is a (possibly reducible) curve contracted by $\phi$ to $p$.

\setcounter{prop}{-1}
\begin{prop}[{\cite[pages 138--139]{FS2}, \cite[Theorem 1.14]{DF}}]\label{thm:intro:equivalence}
 Let $\phi : X \dashto X$ be a rational map $\phi$ on a smooth surface $X$. Then $\phi$ is algebraically stable if and only if it has no destabilising orbits.
\end{prop}

\makeatletter
\let\the@thm\thethm
\let\the@prop\theprop
\let\the@cor\thecor
\let\the@lem\thelem
\makeatother

\renewcommand*{\thethm}{\Alph{thm}}
\renewcommand*{\theprop}{\Alph{prop}}
\renewcommand*{\thecor}{\Alph{cor}}
\renewcommand*{\thelem}{\Alph{lem}}

It was the idea of Diller and Favre \cite{DF} to not settle for an unstable map, but to find a birational change of coordinates $\pi : Y \dashto X$ such that the conjugate map $\psi = \pi \circ \phi \circ \pi^{-1} : Y \dashto Y$ is algebraically stable. We call $\pi : (\psi, Y) \dashto (\phi, X)$ an \emph{algebraic stabilisation}, and $\phi$ \emph{potentially algebraically stable}.  For the reasons above, it is highly desirable to understand when and how algebraic stabilisation can be achieved for a given rational map. 

 Since Diller and Favre showed that birational surface maps can always be stabilised \cite{DF}, there have been few further results in this direction for large classes of maps. There are neither many positive results, nor many counterexamples, and essentially all of them involve monomial or polynomial maps \cite{Fav, FJ11, JW, DL}. 
 A more detailed history is given below.
 
The purpose of this article is to address the matter of algebraic stabilisation for the class of maps in two dimensions called \emph{skew products}; classically these are mappings of the form 
\[\phi : (x, y) \longmapsto (\phi_1(x), \phi_2(x, y)).\]
 More generally, a map $\phi : X \dashto X$ on a surface $X$ is a skew product iff there is a (dominant rational) fibration $h : X \dashto B$ to a curve, and a map $\phi_1 : B \to B$ such that the following diagram commutes:
\[
\begin{tikzcd}
 X \arrow[dashed]{r}{\phi} \arrow[dashed, swap]{d}{h} & X \arrow[dashed]{d}{h} \\
 B \arrow[swap]{r}{\phi_1} & B
\end{tikzcd}
\]

DeMarco and Faber \cite{DeF1, DeF2} studied the special case where $\phi_1 = \id$ on $\mathbb D \times \P^1$, proving that there exists a (possibly singular) surface $\hat X$ and a birational morphism $\pi : (\hat \phi, \hat X) \to (\phi, \mathbb D \times \P^1)$ such that $\hat \phi$ satisfies the geometric criterion for algebraic stability. One might call this a `geometric stabilisation'. 
The class of skew products is an intriguing case for understanding algebraic stabilisation because the dynamical degree of any potentially algebraically stable map is always an algebraic integer, whereas the dynamical degree of a skew product on a surface is always an \emph{integer} \cite{DinhNg, Truong}, \[\lambda_1(\phi) = \max\set{\deg(\phi_1), \deg_y(\phi_2)}.\]
 One might hope therefore that all skew products are potentially algebraically stable. 
In fact, the picture is more complicated, as demonstrated by our main results.

\begin{thm}\label{thm:intro:stabskew}
 Let $h : X \dashto B$ be a birationally ruled surface over a curve $B$, and $\phi : X \dashto X$ be a rational skew product over $B$ such that the base map $\phi_1 : B \to B$ has no superattracting periodic points.
 
 Then there is a smooth surface $\hat X$ and a birational morphism $\pi : (\hat \phi, \hat X) \to (\phi, X)$, blowing up $X$ finitely many times, such that the lift $\hat \phi : \hat X \dashto \hat X$ is algebraically stable.
\end{thm}

The condition that $\phi_1$ has no superattracting cycles turns out to be necessary. 

\begin{thm}\label{thm:intro:skewcounter}
 Consider the map \[\psi : (x, y) \longmapsto \left((1-x)x^2, (1-x)(x^4y^{-3} + y^3)\right)\] as defined on $\P^1 \times \P^1$. 
 There is no birational map $\rho: X \dashto \P^1 \times \P^1$ conjugating $\psi$ or any of its iterates to an algebraically stable map, even if $X$ is allowed to be singular.
\end{thm}
  
We highlight that the surface $\hat X$ in \autoref{thm:intro:stabskew} is smooth. 
 Simultaneously establishing smoothness of $\hat X$ and algebraic stability of $\hat \phi$ is a major objective of this paper that significantly complicates the proof of \autoref{thm:intro:stabskew}. By way of contrast, the (geometric) stabilisation procedure of DeMarco and Faber generally results in a singular surface. In fact, they asked in \cite[Remark 1.4]{DeF2} whether one could obtain a smooth model in general, so \autoref{thm:intro:stabskew} answers their question positively. 
In general, it is unclear whether attaining the geometric criterion for algebraic stability on a singular surface results in an algebraically stable map in the original functorial sense of \autoref{eqn:AS}. Only in the presence of this do certain dynamical calculations simplify. 
Therefore it is important to obtain a geometric stabilisation $\pi : (\psi, Y) \dashto (\phi, X)$ where $Y$ is smooth.

We also stress that our stabilisation procedure succeeds for $\phi$ rather than requiring us to first replace $\phi$ with an iterate $\phi^N$. In the case of monomial maps \cite{Fav} and plane polynomial maps \cite{FJ11} Favre and Jonsson found it necessary to sometimes pass to an iterate to obtain a smooth stabilisation.

Our specific procedure arrives at an algebraic stabilisation $\pi : (\hat \phi, \hat X) \to (\phi, X)$ through a sequence of point blowups on the surface $X$. This guarantees both that $\hat X$ is smooth, and $\pi$ is a birational morphism. Conversely, recall that any birational morphism of smooth surfaces can be decomposed into individual point blowups. The technical argument that the process successfully terminates is in no way effective, and in most cases requires a very large number of point blowups; it may be impractical to implement on a given example. However, the author showed in \cite{stability} that whenever \emph{any} such stabilisation exists, one can construct a \emph{minimal stabilisation} by a much more straightforward process, repeatedly blowing up destabilising orbits until none remain. The corollary below follows directly from \autoref{thm:intro:stabskew} and \cite[Theorem 2]{stability}.

\begin{cor}\label{cor:intro:minskew}
Let $h : X \dashto B$ be a birationally ruled surface and $\phi : X \dashto X$ be a rational skew product over $B$ such that $\phi_1 : B \to B$ has no superattracting periodic points. Then the Minimal Stabilisation Algorithm on $(\phi, X)$ terminates, producing the (unique) minimal stabilisation $\pi : (\hat \phi, \hat X) \to (\phi, X)$.
\end{cor}

There are several major contributions leading to the proofs of \autoref{thm:intro:stabskew}, \autoref{thm:intro:skewcounter} and our understanding of algebraic stability for skew products. We give an overview of these in the following three paragraphs; further discussion will be provided in the relevant sections. 
In \autoref{sec:general} we reduce the problem to analysing $\phi$ on periodic fibres of $X$. 
 The action of $\phi$ on divisors within a (say fixed) fibre $X_b$ induces a map we call a \emph{non-Archimedean skew product} $\phi_* : \P^1_\an(\K) \to \P^1_\an(\K)$ on the Berkovich projective line. Here, each surface $Y$ given by modifying $X$ over the fixed fibre corresponds to a finite subset $\Gamma = \Gamma(Y) \subset \P^1_\an$. Additionally, the lift $\psi : Y \dashto Y$ of $\phi$ can be understood through the action of $\phi_* = \psi_*$ on $\Gamma$ and the components of $\P^1_\an\sm\Gamma$.
 
In \autoref{sec:vertex}, we recount results from \cite[\S 4]{thesis} which make precise the relationship between $Y$ and $\Gamma(Y)$. In particular, we spell out properties of $\Gamma$ that characterise smoothness for $Y$ and relate algebraic stability for the lift $\psi : Y \dashto Y$ to the behaviour of the iterates $\phi_*^n$ on $\Gamma$. 
This is somewhat similar to what was done by Favre and Jonsson in \cite{FJ07, FJ11}. However, the analogue of $\phi_*$ in their setting is essentially a contraction mapping. 
On the contrary, in our context, the corresponding dynamics of $\phi_*$ is often quite chaotic, so we are prevented from using the techniques of Favre and Jonsson.

The second ingredient is a Fatou-Julia theory for skew products such as $\phi_*$, which is developed at length (and in greater generality) in \cite[\S 3]{thesis}. In particular, the classification of Fatou components for a non-Archimedean skew product $\phi_*$ \cite[Theorem M]{thesis} is fundamental to our method to prove \autoref{thm:intro:stabskew}. See \cite{NZ} for an independent parallel treatment. At least when the fixed fibre is not superattracting, the dynamical theory of $\phi_*$ is remarkably close to the one developed by Rivera-Letelier \cite{RL2, RL1} and Benedetto \cite{Bene0, Bene2, Bene3} for rational maps on the Berkovich projective line $\P^1_\an$. However, skew products have a more flexible structure which does not permit use of the algebraic techniques utilised with Berkovich rational maps.
In general, the behaviour of Julia points leads to an explanation of when a skew product is potentially algebraically stable, informing the definition of counterexamples such as \autoref{thm:intro:skewcounter}. This is discussed in \autoref{sec:skewcounter}.

Finally, to prove \autoref{thm:intro:stabskew} in \autoref{sec:thetrickybit}, we perform a rather intricate iterative procedure to manufacture a finite vertex set $\Gamma \subset \P^1_\an$ with two qualities which are difficult to reconcile: A) It has a large enough number of points, distributed geometrically, so that the corresponding surface $\hat X$ is smooth. 
B) Its points iterate well under $\phi_*$, each either remains in $\Gamma$ forever or leaves and never comes close to returning.

Before closing the introduction, we wish to review some previous results with regard to algebraic stabilisation.
Diller and Favre showed \cite{DF} that any birational self-map of a surface $\phi : X \dashto X$ is potentially algebraically stable, moreover their stabilisation is a birational morphism $\pi : \hat X \to X$. 
 Favre \cite{Fav} classified which monomial maps on $\P^2$ can be stabilised, depending on whether the integer $2 \times 2$ matrix defining it is a rational or irrational rotation of the plane; see also Jonsson and Wulcan \cite{JW}. This provided the first negative example, where an algebraic stabilisation does not exist. Later, Diller and Lin \cite{DL} gave geometric criteria for potential stability for maps which preserve a two-form as monomials do. 
 Favre and Jonsson \cite{FJ04, FJ07, FJ11} proved that for each polynomial map $f$ in two variables, there exists a compactification of $\C^2 \hookrightarrow X$ and an $N \in \N$ such that $(f^{N+n})^* = (f^N)^*(f^n)^* =  (f^N)^*(f^*)^n$ for every $n \in \N$; in particular, $f^N$ is algebraically stable. This `eventual' algebraic stability is sufficient for dynamical applications. Following Favre and Jonsson, Gignac and Ruggiero \cite{GR1, GR2} also considered the stabilisation of germs on blowups over a point. More recently, in \cite{Abboud} Abboud extended \cite{FJ07} to endomorphisms of affine surfaces. As mentioned above, DeMarco and Faber showed that maps on $\P^1\times\P^1$ of the form, $(t, R(t, z))$, where $R \in \C(t, z)$ have a (geometric) stabilisation \cite{DeF1, DeF2}. 
 Algebraic stability has been studied for particular families of rational maps by Bedford, Kim, et al.\ \cite{BK06, BKTAM, BK10}, with applications to integrable systems. Algebraic stability for correspondences has been studied, initially by Ramadas \cite{Ram1, Ram2}. Recently, Weinreich showed that an algebraic billiards correspondence is potentially algebraically stable \cite{Weinreich2}.

\makeatletter
\let\thethm\the@thm
\let\theprop\the@prop
\let\thecor\the@cor
\let\thelem\the@lem
\makeatother

\setcounter{thm}{0}


\section*{Acknowledgements}

The author is most grateful to Jeffrey Diller for many fruitful conversations throughout this project, and his expertise in exposition. We also thank Robert Benedetto, Laura DeMarco, Xander Faber, Mattias Jonsson, Nicole Looper, Eric Riedl, Roland Roeder, and Max Weinreich for valuable discussions and comments.

 

{
\section{Background and Preliminaries}\label{sec:background}
}


Assume for the rest of this article that all surfaces are projective varieties over an algebraically closed field $\k$. 
We will use dashed arrows $\phi : X \dashto Y$ to denote a rational map, and a solid arrow $\phi : X \to Y$ to mark one that we are sure is a morphism.

We briefly recall the definitions of indeterminate point, exceptional curve, and destabilising orbit; for a more detailed discussion, see \cite{stability}.

Let $X, Y$ be surfaces and $\phi : X \dashto Y$ a rational map. Let $U$ be the largest (open) set on which $\phi : U \to Y$ is a morphism, then we define the \emph{indeterminate set} as $I(\phi) = X \sm U$. Alternatively, these are the finitely many points at which $\phi$ cannot be continuously defined. After blowing up the indeterminate points finitely many times we obtain a \emph{graph} of $\phi$, $\pi_1 : \Gamma_\phi \to X$ whence $\phi$ lifts to $\pi_2 : \Gamma_\phi \to Y$. Now, for any $p \in X$ one can naturally define its image $\phi(p)$ as $\pi_2(\pi_1^{-1}(p))$; when $p \in I(\phi)$ and $X$ is a surface, this image will be a curve. Similarly we can generalise $\phi^{-1}$ as $\pi_1 \circ \pi_2^{-1}$.
The \emph{proper transform} of a curve $C$ by $\phi$ is $\overline{\phi(C\sm I(\phi))}$. 
An irreducible curve $C \subset X$ is said to be \emph{exceptional} or \emph{contracted} by $\phi$ iff its proper transform is a point in $p \in Y$; in this case $C \subseteq \phi^{-1}(p)$. We define the \emph{contracted set}, $\E(\phi)$, of $\phi$ to be the union of all (the finitely many irreducible) contracted curves in $X$. 

\begin{defn}
 We write $\rho : (\psi, Y) \dashto (\phi, X)$ to indicate that $\rho : Y \dashto X$ is a birational map conjugating $\phi : X \dashto X$ to $\psi = \rho^{-1} \circ \phi \circ \rho : Y \dashto Y$. When $\psi : Y \dashto Y$ is algebraically stable, we say that $\rho$ is a \emph{stabilisation} and that it \emph{stabilises} $\phi$. Given a rational map $\phi : X \dashto X$, we may say that $\phi$ is \emph{potentially algebraically stable} iff a stabilisation exists.
\end{defn}

A contracted curve $C \subseteq \phi^{-1}(p)$ is a \emph{destabilising curve} for $\phi$ iff there is an $n \in \N$ such $\phi^{n-1}(p) \ni q$ and $q \in I(\phi)$. Then we call any irreducible component of $\phi(q)$ an \emph{inverse destabilising curve} of $\phi$. The geometric criterion \autoref{thm:intro:equivalence} for algebraic stability says that $\phi : X \dashto X$ is \emph{algebraically stable} iff there are no destabilising curves for $\phi$. The \emph{destabilising orbit} here is $p, \phi(p)\dots, \phi^{n-1}(p)$ and its \emph{length} is $n$. Often, as in this \thesisarticle{thesis}{article}, we may assume that the terms $p, \phi(p)\dots, \phi^{n-1}(p)$ in a destabilising orbit are closed points, rather than allow e.g.\ $\phi^{n-1}(p)$ to be a curve containing an indeterminate point $q$. As discussed in \cite{stability}, if a destabilising orbit exists of the latter kind then it contains a `minimal' destabilising orbit of closed points. For further discussion on the equivalent geometric criterion for algebraic stability, the author recommends Roeder \cite{roeder}.


\begin{defn}
  We say that $X$ is a \emph{birationally ruled surface over $B$}, iff $X$ is a projective surface (a projective $2$-dimensional variety over $\k$) with a dominant rational map $h : X \dashto B$ such that $h^{-1}(b) \cong \P^1$ for all but finitely many $b \in B$.
\end{defn}

This loose definition turns out to be equivalent to saying that $X$ is birational to the product $B \times \P^1$ where $h$ becomes the projection to the first factor, see \cite[V.2.2]{Hart}.

\begin{defn}
We say $\phi: X \dashrightarrow X$ is a \emph{skew product over $B$} if and only if $\phi$ is a dominant rational map, and there is a dominant rational map with connected fibres $h : X \dashto B$ such that the following diagram of rational maps commutes. If $h$ is a birational ruling then we might emphasise this by saying $\phi$ is a \emph{rational skew product}.
\[
\begin{tikzcd}
 X \arrow[dashed]{r}{\phi} \arrow[swap, dashed]{d}{h} & X \arrow[dashed]{d}{h} \\
 B \arrow[swap]{r}{\phi_1} & B
\end{tikzcd}
\]
\end{defn}

{The map $\phi_1$ must be a morphism since $B$ is $1$-dimensional. We also call $B$ the \emph{base curve} and say $\phi$ is a skew product over $B$.} 
 
\begin{prop}\label{prop:galois:betterskewprod}
 Suppose that $X$ is a surface, $B$ a curve, and $h : X \dashto B$ a rational map such that all but finitely many of the fibres of $h : X \dashto B$ are (possibly disconnected) rational curves. Let $\phi: X \dashrightarrow X$ be a rational map such that the following diagram commutes.
 \[
\begin{tikzcd}
 X \arrow[dashed]{r}{\phi} \arrow[swap, dashed]{d}{h} & X \arrow[dashed]{d}{h} \\
B \arrow[swap]{r}{\phi_1} & B
\end{tikzcd}
\]
Then after replacing $X$ with its smooth desingularisation $\tilde X$, we can also replace $B$ with a smooth curve $\tilde B$ and a fibration $\tilde h : \tilde X \dashto \tilde B$ which is a birational ruling of $\tilde X$ i.e. $\tilde h$ has \emph{connected fibres}, and the induced $\tilde \phi$ is a rational skew product over $\tilde B$. After further blowup, we may assume $\tilde h$ is continuous.
\end{prop}

\begin{proof}
 First we replace $X$ with its smooth desingularisation $\tilde X$. We may further blowup $\tilde X$ until the fibration over $B$ is continuous; for notational simplicity, we will assume this for the rest of the proof. This modification $\rho : \tilde X \to X$ induces a similar diagram of rational maps by conjugation $\tilde \phi = \rho \circ \phi \circ \rho^{-1}$. 
 Now we have a fibration $h' = h \circ \rho : \tilde X \to B$ with $\tilde X$ smooth. 
Then by Stein Factorisation there is a curve $\tilde B$, a morphism $\tilde h : \tilde X \to \tilde B$ and a morphism $g : \tilde B \to B$ such that $\tilde h$ has connected fibres and $g$ is finite. Since $\tilde X$ is smooth and the fibres are connected, $\tilde B$ must already be smooth. Now $\tilde h : \tilde X \to \tilde B$ is a birationally ruled surface and so it has a section $s : \tilde B \to \tilde X$ by Tsen's Theorem \cite[\S V.2]{Hart}. 
 \[
\begin{tikzcd}
 \tilde X \arrow[dashed]{r}{\tilde \phi} \arrow[swap]{d}{\tilde h} & \tilde X \arrow{d}{\tilde h} \\
\tilde B \arrow[swap]{d}{g} \arrow[swap, red]{r}{\tilde \phi_1} \arrow[blue, bend left = 60]{u}{s}& \tilde B \arrow{d}{g}\\
B \arrow[swap]{r}{\phi_1} & B
\end{tikzcd}
\]
Therefore the map on $\tilde B$ we need to construct $\tilde \phi_1 : \tilde B \to \tilde B$ is given by $\tilde h \circ \tilde \phi \circ s$, and the whole diagram above commutes.
\end{proof}

\begin{rmk}
  \autoref{prop:galois:betterskewprod} shows that we may take considerably weaker hypotheses (for \autoref{thm:intro:stabskew}) than that of a smooth surface and a skew product, and arrive at one by blowing up $X$ finitely many times and possibly replacing our base curve $B$. We can guarantee that such $X$ and $B$ are smooth, and $h : X \to B$ is continuous with connected fibres. Unless otherwise stated, we will assume this is the situation for all rational skew products for the remainder of this \thesisarticle{chapter}{article}.
\end{rmk}
 
 \begin{lem}\label{fibretofibre}
 Let $\phi : X \dashto X$ be a skew product over $B$.
\begin{itemize}
 \item If $\phi$ contracts the curve $C$ in $X$, then $C \subseteq h^{-1}(z)$ for some $z \in B$.
 \item If $\phi(p) = C$ (i.e.\ $p \in I(f)$ is indeterminate), then $C \subseteq h^{-1}(z)$ for some $z \in B$.
\end{itemize}
 \end{lem}
 
\begin{proof}
 Suppose $\phi$ contracts the curve $C$ in $X$. Either $h(C) = z$ or $h(C) = B$, since $h$ is proper. 
 Let $\phi(C \sm I(\phi)) = p$ and $w = h(p) \in B$, then \[w = h(p) = h(\phi(C)) = \phi_1(h(C)) = \phi_1(B).\] Therefore $h(\phi(p)) = \phi_1(h(p)) = w$ for every $p \in X$, i.e.\ $\phi(X) \subseteq h^{-1}(w)$, so $\phi$ is not dominant $\contra$.
 
 Similarly, if $\phi(p) = C$ and $h(C) = B$, then $h(C) = h(\phi(p)) = \phi_1(h(p))$, a closed point in $B$ (given $\phi_1, h$ are continuous).
\end{proof}

\autoref{fibretofibre} states that a skew product has all its exceptional curves contained in fibres. This prompts the next definition.

\begin{defn}
 Let $\phi : X \dashto X$ skew product over the fibration $h : X \to B$, let us write $\check I(\phi)$ for $h(I(\phi))$ and $\check \E(\phi)$ for $h(\E(\phi))$
\end{defn}

\begin{prop}
Let $\phi : X \dashto X$ be a rational skew product over $B$. If $C$ is a destabilising curve for a $\phi$, then $h(C) = b \in \check \E(\phi)$ and $\phi_1^n(b) \in \check I(\phi)$ for some $n \in \N$.
\end{prop}

\begin{proof}
 By definition, $b \in \check \E(\phi)$ and we have an $n$ such that $\phi^n(C) = p$ and $\phi(p) = D$, a curve in $X$. Therefore \[\phi_1^n(b) = \phi_1^n(h(C)) = h(\phi^n(C)) = h(p).\] 
 This point is in $\check{I}(\phi)$ since $p \in I(\phi)$.
\end{proof}

\begin{rmk}
We may assume that if $B$ is not rational then $\phi_1 : B\to B$ has no ramification points. When the genus of $B$ is at least $1$, Riemann-Hurwitz shows that $\phi_1$ must be unramified, regardless of degree.
 Further, if the genus is $g > 1$ then the Riemann-Hurwitz formula and Hurwitz automorphism theorem show that $\phi_1 : B \to B$ is an automorphism of finite order; i.e., after replacing $\phi$ by an iterate, we may assume $\phi_1$ is the identity.
\end{rmk}

{For a significant part of this article we will use the (dynamical) theory of \emph{non-Archimedean skew products} on the Berkovich projective line, developed in the author's thesis. For basic information on the Berkovich projective line we refer to \cite{Bene}. We primarily refer the reader to \cite[\S3]{thesis} for details, but provide below an extremely brief outline of the most important and relevant features.

 \begin{defn}
  Let $K$ be a non-Archimedean field and $\Psi$ an endomorphism of $K(y)$ extending an automorphism of $K$, i.e. the following diagram commutes:
  \[
\begin{tikzcd}[ampersand replacement = \&]
 K(y) \& K(y) \arrow[swap]{l}{\Psi} \\
K \arrow[hook]{u} \& K \arrow[hook]{u} \arrow{l}{\Psi_1}
\end{tikzcd}
\]
In this case we will call $\Psi : K(y) \to K(y)$ a \emph{skew endomorphism} of $K(y)$. We will typically denote the restriction $\left.\Psi\right|_K$ by $\Psi_1$.
\end{defn}

\begin{defn}[Non-Archimedean Skew Product]
 Suppose that $\Psi : K(y) \to K(y)$ is a skew endomorphism of $K(y)$ and there is a $\q$ such that $\abs{\Psi(a)} = \abs{\Psi_1(a)} = \abs{a}^{\frac 1\q}$ for every $a \in K$. Then we say $\Psi$ is \emph{\dilating} with \emph{scale factor} $\q$. Given such a $\Psi$, we define $\Psi_*$, a \emph{(non-Archimedean) skew product over $K$}, as follows.
\begin{align*}
 \Psi_* : \P^1_\an(K) &\longrightarrow \P^1_\an(K)\\
 \zeta &\longmapsto \Psi_*(\zeta)\\
 \text{where } \norm[\Psi_*(\zeta)]{f} &= \norm[\zeta]{\Psi(f)}^\q
\end{align*}
If $\q=1$ then we call $\Psi_*$ a \emph{simple} skew product. Otherwise, if $\q < 1$ we say it is \emph{superattracting}, and if $\q > 1$ we may say it is \emph{superrepelling}.
\end{defn}

Suppose we have a skew product $\phi : X \dashto X$ over a birationally ruled surface $h : X \to B$ as previously defined, and a fixed fibre of $X$ over $b = \phi_1(b)$. One can complete the local ring at $b \in B$ to $\k[[x]]$ and think of $\phi$ as $\phi(x, y) = (\phi_1(x), \phi_2(x, y))$. This information is equivalent to a $\k$-algebra homomorphism on $\k[[x]](y)$. We can write $\phi_1^*(x) = \phi_1(x) \in \k[[x]]$ where $\phi_1(x) = \lambda x^n + \bO(x^{n+1})$ with some $\lambda \in \k^\times$. This extends to a \emph{\dilating{} skew endomorphism} on its algebraic closure $\K = \K(\k)$, called the \emph{Puiseux series}. It is also helpful extending to its completion, the Levi-Civita field $\hk$, but we will generally forget this formality. We define the \emph{relative degree} as $\rdeg(\phi) = \deg_y(\phi_2)$.
  \[
\begin{tikzcd}[ampersand replacement = \&]
 \k((x))(y) \& \k((x))(y) \arrow[swap]{l}{\phi^*} \\
\k((x)) \arrow[hook]{u}{h^*} \& \k((x)) \arrow[hook, swap]{u}{h^*} \arrow{l}{\phi_1^*}
\end{tikzcd}
\]
This induces a non-Archimedean skew product $\phi_* : \P^1_\an(\K) \to \P^1_\an(\K)$. We may call such a map a \emph{$\k$-rational skew endomorphism}, or say it is \emph{defined over $\k((x))$} to highlight that $\phi_1 \in \k[[x]], \phi_2 \in \k((x))(y)$. 
In this concrete case, with $n$ as above, $\phi_*$ has scale factor $\q = \frac 1n$. So $\phi_*$ is simple iff $n = 1$, and superattracting otherwise, corresponding with the local behaviour at $b \in B$.

Skew products on the Berkovich projective line turn out to be piecewise-linear, proper, open mappings, which preserve the `Types' of points on $\P^1_\an$.

\begin{defn}
Let $\phi_*$ be a skew product. We say an open set $U \subseteq \P^1_\an$ is \emph{dynamically stable} under $\phi_*$ iff $\displaystyle \bigcup_{n \ge 0} \phi_*^n(U)$ omits infinitely many points of $\P^1_\an$.

The \emph{(Berkovich) Fatou set} of $\phi_*$, denoted $\F_{\phi, \an}$, is the subset of $\P^1_\an$ consisting of all points $\zeta \in \P^1_\an$ having a dynamically stable neighbourhood.

The \emph{(Berkovich) Julia set} of $\phi_*$ is the complement $\J_{\phi, \an} = \P^1_\an \sm \F_{\phi, \an}$ of the Berkovich Fatou set.
\end{defn}

A \emph{Fatou component} is a connected component of the Fatou set. The usual properties of Fatou and Julia sets hold, such as $\phi_*(U)$ is a Fatou component whenever $U$ is. For the purposes of this article, we will say an \emph{attracting component} is the immediate basin of attraction for an attracting cycle of classical points, and an \emph{indifferent component} is an affinoid $U \subsetneq \P^1_\an$ such that $\phi_*^n(U) = U$ for some $n$. Following Rivera-Letelier \cite{RL1, RL2, RL4}, in \cite[\S3]{thesis} we proved a classification of Fatou components for skew products. We state it in a simpler form below.

\begin{thm}[Classification of Fatou Components over $\K$, {\cite[Theorem M]{thesis}}] \leavevmode\\
Let $\phi_* : \P^1_\an(\K) \to \P^1_\an(\K)$ be a simple skew product defined over $\k((x))$ of relative degree $d \ge 2$, and let $U \subset \F_{\phi, \an}$ be a periodic Fatou component. Then $U$ is either an indifferent component or an attracting component, but not both.
\end{thm}

As mentioned in the introduction, this will be fundamental to our method for proving \autoref{thm:intro:stabskew}. Also important is the generalisation of Benedetto's `no wandering domains' results \cite[Theorem 11.2, Theorem 11.23]{Bene} to skew products. A Fatou component $U \subset \F_{\phi, \an}$ of $\phi_*$ is \emph{wandering} iff the iterates $U, \phi_*(U), \phi_*^2(U), \dots$ are all distinct. If not, then $U$ is preperiodic, meaning some $\phi_*^n(U)$ is a periodic attracting or indifferent component.

\begin{defn}
 Let $\phi_*$ be a simple skew product of relative degree $d \ge 2$, let $\F_{\phi, \an}$ be the Berkovich Fatou set of $\phi_*$, and let $\zeta \in \P^1_\an$ be a Type II periodic point of $\phi_*$ of minimal period $p$. We say that a wandering component $U$ of $\F_{\phi, \an}$ is in the \emph{attracting basin} of $\zeta$ if there is some integer $N \ge 0$ such that for all $n \ge 0$,
$\phi_*^{N+np}(U)$ is a residue class at $\zeta$.
\end{defn}

The following is a modification of \cite[Theorem 11.23]{Bene}. It leads to a powerful corollary that every Julia Type II point is preperiodic.

\begin{thm}[No Wandering Domains over $\K$, {\cite[Theorem 3.97, Theorem 3.99]{thesis}}]\label{thm:skewstab:wanderingdombasin}\leavevmode\\
Let $\phi_*$ be a simple skew product defined over $\k((x))$. Let $U \subseteq \F_{\phi, \an}$ be a wandering domain of $\phi_*$. Then $U$ lies in the attracting basin of a Type II Julia periodic point.
\end{thm}

\begin{cor}[{\cite[Corollary 3.100]{thesis}, \cite[Proposition 3.9]{DeF2}}]\label{cor:skewstab:typeiijuliaperperiodic}
 Let $\phi_*$ be a simple skew product defined over $\k((x))$. Then any Type II Julia point is preperiodic.
\end{cor}
}

{
\section{General Cases}\label{sec:general}

In this section we commence the proof of \autoref{thm:intro:stabskew}. First we will deal with two general cases, boiling the issue down to periodic fibres.
}

\begin{thm}\label{thm:skewstab:nonperiodic}
 Let $h : X \dashto B$ be a fibration over a curve $B$ and let $\phi : X \dashto X$ be a skew product. 
 Then there is a smooth surface $\tilde X$ and a birational morphism $\pi_0 : (\tilde \phi, \tilde X) \to (\phi, X)$, blowing up $X$ finitely many times, such that all destabilising orbits (points and curves) of the conjugate $\tilde \phi : \tilde X \dashto \tilde X$ are contained in periodic fibres.
\end{thm}

\begin{proof}
Throughout this proof we will refer to the same map even after modifying the surface, to avoid a landslide of notation. In other words, after several blowups via $\rho$ we would usually have $\rho : (\psi, Y) \to (\phi, X)$, but we will identify $\psi$ with $\phi$.

\noindent\textbf{Wandering destabilising orbits.} First, suppose that $b \in \check I(\phi)$ is a point with infinite forward orbit. Since $\check I(\phi)$ is finite, we can replace $b$ with the last such point in $\Orb^+_{\phi_1}(b)$. Since $\check \E(\phi)$ is finite and $b$ is not periodic, the points in $\check \E(\phi)$ only appear finitely many times in the backward orbit of $b$. Let $b_1, \dots, b_n$ be those points. Now blowup every indeterminate point in $h^{-1}(b)$ until the new version of $\phi$ is continuous on $X_b$. This decreases the maximum length of orbits between the $b_j$ and $\check I$ by one. In doing this we may create a new exceptional curve $C$ in $X_b$, but because $b$ is wandering the forward orbit of $b \in \check \E(\phi)$ is disjoint from $\check I(\phi)$, hence $C$ is not destabilising. Now repeat this process for each $b' \in \phi_1^{-1}(b)$, and so on, each time reducing the maximum length of destabilising orbits over wandering points of $\phi_1$ in $B$.
Therefore, by induction, eventually every destabilising orbit lies within fibres above preperiodic points in $B$.

\noindent\textbf{Preperiodic destabilising orbits.} 
Suppose that $c_1, c_2, \dots, c_N$ is a cycle for $\phi_1$. For each $j$ there may be finitely many $b \in \check \E(\phi) \cap \Orb^-(c_j)\sm \set[c_j]{1 \le j \le N}$. These $b$ are not periodic themselves. Consider what happens if we blowup points in the image of the exceptional curves above such $b$ until they are not exceptional anymore. We begin the process with each $b_0 \in \Orb^-(c_j)$ such that $\Orb^-(b_0) \cap \check \E(\phi) = \emp$. This operation is finite in each fibre; it may create more indeterminacy but not in the forward orbit of exceptional curves, so any newly created destabilising orbit projects down to \[\phi_1(b_0), \phi_1^2(b_0), \dots, \phi_1^m(b_0) = c_j, c_{j+1} \dots.\] Therefore in one step we have reduced the length of any such projection of a destabilising orbit (prior to the $c_j$ cycle) from $m$ to $m-1$. Continuing this process, we can push all destabilising orbits into (not preperiodic but) \emph{periodic fibres}.
\end{proof}

\noindent\textbf{Periodic destabilising orbits.} 
\autoref{thm:skewstab:nonperiodic} has reduced the proof of {\autoref{thm:intro:stabskew}} to the following theorem. We will expand on this in the next section.

\begin{thm}\label{thm:skewstab:periodicgeom}
 Let $h : X \to B$ be a smooth birationally ruled surface and $\phi : X \dashto X$ be a rational skew product over $B$. Suppose that $b_1, b_2, \dots b_N$ is a cycle for $\phi_1$, none of which are critical. Then there is a smooth global model $\hat X$ over $(b_j)_{j=1}^N$ dominating $X$, such that the conjugate $\hat \phi : \hat X \dashto \hat X$ is algebraically stable.
\end{thm}


\thesisarticle{\subsection{Reduction to Vertex Sets}}{\section{Reduction}\label{sec:vertex}}

In what follows, we translate and further reduce \autoref{thm:skewstab:periodicgeom} to one about vertex sets for non-Archimedean skew products on the Berkovich Projective Line, namely \autoref{thm:extendvertexset}. We start by summarising some concepts from {\cite[\S 4.11]{thesis}}.

{%
Let $\phi : X \dashto X$ be a skew product on a birationally ruled surface $h : X \to B$ over the base field $\k$ (e.g.\ $\C$). Let $b \in B$ and consider its local ring $\bO_{B, b}$ on $B$, this is a discrete valuation ring, and let $m_b = (x)$ be its maximal ideal. The fraction field is $\Frac(\bO_{B, b}) = \k(B)$ and the residue field is $\k = \bO_{B, b}/m_b = \k(b)$. The associated \emph{order of vanishing norm} $\abs\cdot$ with respect to $b$ and $x$ measures the order of vanishing of functions on $B$ at $b$. This norm makes $(\k(B), \abs\cdot)$ a non-Archimedean field with ring of integers $\bO_{B, b}$. 
Assuming $B$ is smooth at $b$, then by the Cohen structure theorem, the completion of $\bO_{B, b}$ is isomorphic to $\k[[x]]$, where the generator $x$ may be considered the same as above.

\begin{defn}
 Let $h : X \to B$ be a birationally ruled surface, $b \in B$. A \emph{global model of $X$ over $b \in B$} is a birationally ruled normal (but possibly singular) surface $g : Y \to B$ which is isomorphic to $X$ away from $X_b$. Meaning there is a birational map $\iota : Y \dashto X$ such that $\iota : Y \sm Y_b \to X \sm X_b$ is an isomorphism over $B$. A birational map $\rho : Y \dashto Y'$ of models over $b$ is a birational map over $B$ such that $\iota' \circ \rho_\ell \circ \iota^{-1}$ restricts to the identity on $X\sm X_b$.
 
Further, given finitely many closed points $b_1, \dots, b_n \in B$, we make a similar definition for a \emph{global model of $X$ over $(b_j)$} where the map $\iota$ is an isomorphism away from $\bigcup_j X_{b_j}$.
\end{defn}

Given a global model $Y$ of $X$ over $b \in B$, we can define a \emph{reduction map} $\redct_{Y, b} : \P^1_\an(\K) \to Y_b$ which maps each $\zeta \in \P^1_\an$ to the point (possibly a curve!)\ cut out by functions $f$ such that $\norm[\zeta]{f} < 1$, i.e.\ those which vanish according to $\zeta$. Further, this reduction factors as $\redct_{Y, b} = \mathfrak p \circ \redct'_{Y, b}$ through the natural quotient $\mathfrak p : \P^1_\an(\K) \to \V = \P^1_\an(\k((x)))$ by the action of the Galois group $G = \Gal(\K/\k((x)))$, to the valuative tree. Through $\redct'_{Y, b}$, there is a one-to-one correspondence between the finitely many irreducible curves in $Y_b$ (its generic points) and a finite set of `divisorial' or Type II points $\Gamma_G(Y) \subset \V$. Further pulling this back to $\P^1_\an(\K)$ we obtain a finite $G$-invariant subset $\Gamma(Y) \subset \P^1_\an(\K)$. 
See {\cite[4.48]{thesis} or \cite[2.4.4]{Berk}}.

\begin{defn}
 Let $\Gamma \subset \P^1_\an$ be a finite set of Type II points --- which we will call a \emph{vertex set}. Then $\P^1_\an\sm \Gamma$ is the disjoint union of a collection $\mathcal{S}(\Gamma)$ of open connected affinoids, each of which we call a $\Gamma$\emph{-domain}. If a $\Gamma$-domain has one boundary point, we call it a $\Gamma$\emph{-disk}, and if it has two, we call it a $\Gamma$\emph{-annulus}. 
Let $\K = \K(\k)$ be the Puiseux series in $x$ over $\k$, and $G = \Gal(\K/\k((x)))$. If $\Gamma \subset \P^1_\an(\K)$ is $G$-invariant, then projecting to $\P^1_\an(\k((x)))$ we obtain a vertex set, denoted by $\Gamma_G$, and $\Gamma_G$-domains $\mathcal{S}_G(\Gamma)$. 
 We let $\mathcal{S}^+(\Gamma) = \mathcal{S}(\Gamma) \cup \Gamma$ be the set of $\Gamma$-domains and the points in $\Gamma$ itself. 
 Given a global model $Y$ over $b$, we define $\Gamma(Y)$ to be the vertex set $\redct_{Y, b}^{-1}(Y_{\text{gen}})$.
\end{defn}

One can reconstruct the \emph{dual graph} $\Delta(\Gamma)$ whose vertices are $\Gamma$ and taking an edge $\zeta\xi$ whenever $\zeta, \xi \in \partial U$ for some $\Gamma$-domain $U$. When $\Gamma_G = \Gamma_G(Y)$ for a global model $Y$ over $b$, $\Delta(\Gamma_G)$ is precisely the dual graph of divisors in $Y_b$, where edges signify intersections. Clearly if $U$ has more than two boundary points, then $\Delta(\Gamma)$ contains a triangle, so if $\Delta(\Gamma)$ is a tree then every $\Gamma$-domain must be a disk or annulus; in our context of rational fibres, the converse also holds. This corresponds to a model with simple crossings (SC) in the fibre; see \cite[Proposition 4.51]{thesis}.
}

\subsection{Smoothness}

Next we outline an equivalence between smoothness of a global model $Y$, and the geometry of its associated vertex set $\Gamma(Y)$. 
\thesisarticle{}{For the latter we recall the (equivalent) definitions of multiplicities from \cite[\S4]{thesis}.}

\begin{defn}
 Let $\zeta \in \P^1_\an(\K)$. Define $\mm(\zeta)$ to be $|\Orb_G(\zeta)|$. For a subset $U \subset \P^1_\an$ define $\mm(U) = \min_{\zeta \in U} \mm(\zeta)$. Define the \emph{multiplicity $n$ subtree} by \[\T_n = \set[\zeta \in \P^1_\an(\hk)]{\mm(\zeta) \divides n}.\]
\end{defn}

\begin{prop}[{\cite[Propositions 4.3, 4.20]{thesis}}]
Let $\zeta \in \P^1_\an$.
\begin{enumerate}[label = (\roman*)]
\item If $\zeta = \gamma \in \hk$ is Type I, then $\mm(\gamma)$ is the smallest integer $m$ such that $\gamma \in \k((x^{\frac{1}{m}})) < \K$, or $\infty$ otherwise. Equivalently, if $\gamma \in \hk\sm\K$, then $\mm(\gamma) = \infty$, else if $\gamma \in \K$, then it has a degree $\mm(\gamma)$ minimal polynomial over $\k((x))$.
 \item Suppose $\zeta = \zeta(\gamma, r)$ is Type II or III and let $a$ be the Puiseux series obtained by truncating any Puiseux series $b \in \CD(\gamma, r)$ to $\bO(r)$. Then $\mm(a) = \mm(\zeta) = \min_{b \in \CD(a, r)} \mm(b)$.
 \item If $\zeta$ is Type IV, then $\mm(\zeta) = \infty$.
\end{enumerate}
\end{prop}

It follows that $\T_m \subseteq \T_n \iff m \divides n$, and $\T_n$ is indeed a closed connected set. Further, $\mm{} : \P^1_\an(\hk) \longrightarrow \N_+\cup\{\infty\}$ is lower semicontinuous, both in the usual order on $\N$, and with respect to the multiplicative order of natural numbers $(\N_+\cup\{\infty\}, <_m)$. In {\cite[Proposition 4.25]{thesis}} we outline the structure of the multiplicity $n$ subtree. 

\begin{prop}[{\cite[Proposition 4.25]{thesis}}]\label{prop:skewstab:multsubtreevalency}
 The subtree $\T_n$ is an infinite tree with discrete branching in the following sense: every (non-endpoint) vertex $\zeta \in \T_n$ of valency at least $3$ is of Type II and in every direction at $\zeta$ there is an edge of length $1/n$ which has no further branching.
\begin{itemize}
 \item The set of non-endpoint vertices is of the form 
 \[\set[\zeta \in \bH]{\zeta = \zeta(a, |x|^\frac pq),\ \mm(a), q \divides n},\]
 hence $d_\bH(\zeta_1, \zeta_2) \in \tfrac 1n \N$ for any two $\zeta_1, \zeta_2$ in the set.
 \item  Let $\zeta = \zeta(a, |x|^\frac pq)$ with $\mm(a) = \mm(\zeta) = m$, $\GCD(p, q) = 1$, and set $\mg = \LCM(m, q)$. 
\begin{enumerate}[label = (\roman*)]
 \item There is a $\zeta' \in (\zeta, \infty]$ such that $\mm(\xi) = m$ for every $\xi \in [a, \zeta')$.
 \item Let $c \in \C^*$. Then for every $\xi \in \vec v(a + cx^\frac pq)$, we have $\mm(\xi) \ge \mg$, and $\mm(\xi) = \mg$ for every $\xi \in [a + cx^\frac pq, \zeta)$.
 \item In particular, $\zeta$ has two directions with possibly lower multiplicities, $\mm(\vec v(\infty)) = 1$, $\mm(\vec v(a)) = m$, and for every other direction $\mm(\bvec v) = \mg$.
\end{enumerate}
\end{itemize}
\end{prop}

\begin{defn}
 Let $\zeta \in \P^1_\an$. Define \emph{generic multiplicity}, $\mg(\zeta)$ to be
\begin{enumerate}
 \item $\mm(\zeta)$ if $\zeta$ is Type I,
 \item $\mg$ as in \autoref{prop:skewstab:multsubtreevalency} if $\zeta$ is Type II,
 \item $\infty$ if $\zeta$ is Type III or IV.
\end{enumerate}
\end{defn}

We find that $\mg(\zeta)$ is the smallest $n$ such that $\zeta$ is a vertex in $\T_n$ if at all, or $\infty$ otherwise. The set of vertices of $\T_n$ is given by $\set[\zeta \in \P^1_\an]{\mg(\zeta) \mid n}$. For the purpose of stabilisation and producing smooth models, it will be useful to instead make a similar definition $\GV_n = \set[\zeta \in \P^1_\an]{\mg(\zeta) \le n}$. This latter set gives the vertices of $\bigcup_{m \le n}\T_m$. 

\begin{defn}
Let $\zeta$ be a Type II point and $\bvec v \in \Dir \zeta$. We will say $\bvec v$ is a \emph{generic direction} iff $\mm(\bvec v) = \mg(\zeta)$, and say it is \emph{special} otherwise.
\end{defn}

\begin{defn}
We say a vertex set $\Gamma \subset \P^1_\an$ is \emph{geometric} iff it is (Galois) $G$-invariant, i.e. it lifts to a $\Gamma_G \subset \V$.
Further, we say that $\Gamma$ is \emph{smooth} if and only if for every $\Gamma$-domain $U \in \mathcal{S}(\Gamma)$, either
\begin{enumerate}
 \item $U$ is a disk with boundary point $\zeta$ and $\mm(U) = \mg(\zeta)$; or
 \item $U$ is an annulus with boundary points $\zeta_1, \zeta_2$ and $\mg(\xi) > \max(\mg(\zeta_1), \mg(\zeta_2))$ for all $\xi \in U$. 
\end{enumerate}
 \end{defn}

In essence, the data of a vertex set associated to a smooth global model $Y_b$ is exactly what one should expect from repeated blowing up of a minimal smooth model. For this reason it is important to differentiate between `free' and `satellite' exceptional curves. See {\cite[Theorem 4.68]{thesis}}. 
 
 \def\unflanked{unflanked}
 \def\flanked{flanked}
 
\begin{defn}
 Let $\zeta \in \P^1_\an$ be Type II.
\begin{itemize}
 \item When $\mg(\zeta) = 1$ we say $\zeta$ is \emph{integral}.
 \item We say $\zeta$ is \emph{free} iff $\mg(\zeta) = \mm(\zeta)$, and \emph{satellite} otherwise.
\item Suppose $\Gamma \ni \zeta$ is a vertex set. We will say $\zeta$ is \emph{\flanked{}} (by $\Gamma$) iff $\Gamma$ intersects each of its special directions, and \emph{\unflanked{}} otherwise.
 \end{itemize}
 \end{defn}
 
 \begin{rmk}\label{rmk:skewstab:specialgeneric}
 \autoref{prop:skewstab:multsubtreevalency} says that a Type II point $\zeta$ has at most two special directions, namely $\vec v(a), \vec v(\infty)$ where $a \in \K$ is described in the proposition. Further, there are exactly two special directions if and only if $\zeta$ is satellite. Otherwise, $\zeta$ is free and $\mm(\zeta) = \mg(\zeta) = \mm(a) = \mm(\vec v(a))$. The direction $\vec v(\infty)$ is always special unless $\zeta$ is integral, in which case every direction is generic. 
 If $\Gamma$ is geometric, then $\Hull(\Gamma)$ contains a point of multiplicity $1$. 
  If $\Gamma$ contains a point of multiplicity $1$, then any free point is \flanked{}. 
  For a satellite point $\zeta$ to be \flanked{} the second (finite) special direction must intersect $\Gamma$. 
  It turns out that every point of a smooth vertex set is necessarily \flanked{}. Although this condition is not sufficient, \unflanked{} points will be the main obstruction to smoothness. The first condition of smoothness says precisely that any $\Gamma$-disk is a generic direction.
\end{rmk}
 
 \begin{defn}
 We define $\GV_n$ to be the set of Type II points $\xi$ of generic multiplicity $\mg(\xi) \le n$. Let $\Gamma \subset \GV_n$ be a vertex set and $n \in \N$ below.
 \begin{itemize}
 \item Define the \emph{$n$-convex hull} of $\Gamma$ to be $\Hull(\Gamma) \cap \GV_n$. We will say $\Gamma$ is \emph{$n$-convex} if it equals its $n$-convex hull.
 \item We say $\Gamma$ is \emph{smoothly $n$-convex} iff it is geometric, $n$-convex, and each of its points is \flanked{}.
 \item Define the \emph{smooth $n$-convex hull} of $\Gamma$ to be the smallest smoothly $n$-convex set containing $\Gamma$.
\end{itemize}
\end{defn}

\begin{rmk}[Warning]
The smooth $n$-convex hull of $\Gamma$ may not be contained in the convex hull, $\Hull(\Gamma)$, of $\Gamma$.
\end{rmk}

\begin{prop}\label{prop:skewstab:smoothhull}
 Let $\Gamma \subset \GV_n$ be a geometric vertex set. 
 Then for each \unflanked{} point $\zeta \in \Hull(\Gamma) \cap \GV_n$ there exists a $\xi_1 \in \Hull(\Gamma)$ and a free $\xi_2 \nin \Hull(\Gamma)$ (both lying in special directions) such that $\zeta \in (\xi_1, \xi_2)$, $\mg(\xi_1) \mid \mm(\zeta) = m = \mg(\xi_2) \le n$, and $d_\bH(\xi_1, \xi_2) = 1/m$.
 In particular if $\Lambda \supseteq \Gamma$ is the collection of all $\zeta\in\Gamma$ and associated such $\xi_2$ as above, then the smooth $n$-convex hull of $\Gamma$ is $\Hull(\Lambda) \cap \GV_n$.
\end{prop}

\begin{proof}
Suppose $\zeta \in \Gamma$ is a vertex of $\Hull(\Gamma)$ which is not \flanked{}. By \autoref{prop:skewstab:multsubtreevalency} and \autoref{rmk:skewstab:specialgeneric} we may assume $\zeta$ is satellite with special directions $\vec v(a), \vec v(\infty)$ such that $\mg(\zeta) > \mm(\zeta) = \mm(\vec v(a)) = m$ for some Type I point $a$. 
It must be that $\Gamma$ is disjoint from (exactly) one of these two directions. The non-endpoint vertices of $\T_m$ (which contains $[a, \infty] \ni \zeta$) are Type II points with generic multiplicity $m' \mid m$, and separated by edges of hyperbolic length $1/m$. Given that $\mg(\zeta) > m$, $\zeta$ lies on such an edge $(\xi_1, \xi_2)$ where $\xi_1 \in \Hull(\Gamma)$, $\xi_2 \nin \Hull(\Gamma)$ lie in the special directions, and $\mg(\xi_1), \mg(\xi_2) \mid m$. We claim that $\mg(\xi_2) = m$. If $\mg(\xi_2) = m' < m$ then $\zeta \in (\zeta_1, \xi_2)$ where $\zeta_1$ is some point in $\Gamma$ of generic multiplicity $1$. Then $\zeta \in \T_{m'}$ by convexity of the multiplicity $m'$ subtree, so $\mm(\zeta) < m\ \contra$. 
Consider a Galois conjugate $g_*(\zeta)$ of $\zeta$. Then $g_*(\xi_1), g_*(\xi_2)$ satisfy the conclusion of the proposition for $g_*(\zeta)$; thus the set $\Lambda$ is geometric. The last part follows because every point in $\Lambda$ is \flanked{} and any smooth $n$-convex hull is in the convex hull of its \flanked{} points.
\end{proof}

The purpose of the smooth convex hull, via the following proposition, is to give us a simply and uniformly defined (but rather overkill) target when we attempt to expand $\Gamma$ to a smooth (analytically stable) vertex set.
 
\begin{prop}\thesisarticle{[\autoref{prop:galois:freesat}]}{}\label{prop:skewstab:smoothconvex}
 Every smoothly convex vertex set is smooth.
\end{prop}

\begin{proof}
 Let $\Gamma = \Hull(\Gamma) \cap \GV_n$ be a smoothly convex vertex set. By definition, this is geometric. Let $U \in \mathcal{S}(\Gamma)$ be a $\Gamma$-domain. Suppose $U$ had three (or more) boundary points $\xi_1, \xi_2, \xi_3$; one can see that there is a Type II point $\zeta \in U$ where these appear in three distinct directions. 
 Using \autoref{prop:skewstab:multsubtreevalency} we deduce that $\mg(\zeta) \le \max_i \mm(\xi_i) \le n$; hence $\zeta \in \Hull(\Gamma) \cap \GV_n\ \contra$. 
 Suppose $U$ is an annulus bounded by $\xi_1, \xi_2 \in \Gamma$ and let $\zeta \in U$ be arbitrary. Let $\xi$ be the Type II point with $\zeta, \xi_1, \xi_2$ in distinct directions. Then $\xi \in (\xi_1, \xi_2)$ has generic multiplicity greater than $n$ since $\xi \nin \Gamma = \Hull(\Gamma) \cap \GV_n$. By \autoref{prop:skewstab:multsubtreevalency}, $\vec v(\xi_1), \vec v(\xi_2)$ must be the two special directions with points of multiplicity at most $\mg(\xi)$, and any other direction $\bvec v$, such as $\vec v(\zeta)$ has multiplicity $\mm(\bvec v) = \mg(\xi)$. Therefore $\mm(\zeta) \ge \mg(\xi) > n \ge \max\set{\mg(\xi_1), \mg(\xi_2)}$, as required for smoothness. 
 Finally, suppose that $U$ is a disk with $\partial U = \zeta \in \Gamma$, so consider $U$ as a direction $\bvec v$ at $\zeta$. If $\bvec v$ is special then $\zeta$ is not \flanked{} in $\Gamma$ by definition. Otherwise, it is generic and $\mm(\bvec v) = \mg(\zeta)$, as required for smoothness.
\end{proof}

\begin{thm}[{\cite[Theorem 4.69]{thesis}}]\label{thm:skewstab:smoothness}
Let $h : X \to B$ be a birationally ruled surface, $b \in B$, and let $\Gamma \subset \P^1_\an$ be a vertex set, i.e.\ a finite set of Type II points. Then $\Gamma$ is smooth if and only if there is a smooth global model $Y$ of $X$ over $b$ such that $\Gamma(Y) = \Gamma$.
\end{thm}

By \thesisarticle{\autoref{thm:galois:geometric} and \autoref{prop:galois:vertexextension}}{\autoref{thm:skewstab:smoothness} and \cite[4.50]{thesis}} (compare \cite[Theorem 4.11]{BPR}, \cite[Proposition 3.6]{BFJ}), choosing a sequence of blowups $\rho : Y \to X$ centred in fibres over $b_1, \dots, b_N$ corresponds exactly to finding supersets $\Gamma'_{(j)} \supset \Gamma_{(j)} = \Gamma(X_{b_j})$ for each $1 \le j \le N$ which are \emph{smooth}.


\subsection{Skew Product Correspondence}

Suppose that $\phi : X \dashto X$ is a skew product on the birationally ruled surface $h : X \to B$ and $\phi_1(b) = c$, meaning $\phi$ maps $X_b$ to $X_c$. Then the reduction map induces a mapping $\phi_*$ between analytifications of the two fibres, i.e.\ between Berkovich projective lines.

 \[
\begin{tikzcd}
\P^1_\an \arrow[two heads, swap]{d}{\redct_{X, b}} \arrow{r}{\phi_*} & \P^1_\an \arrow[two heads]{d}{\redct_{X, c}} \\
 X_b \arrow[swap]{r}{\phi}& X_c
\end{tikzcd}
\]

More precisely, we can take the completion of $B$ around $b$ and $c$ such that both points are represented by $(x)$ in $\k[[x]]$, which is isomorphic to each completed local ring. Now $\phi_1$ is locally given by \[\phi_1(x) = \lambda_1x + \lambda_2x^2 + \lambda_3x^3 + \cdots \] where $\lambda_n \in \k$. In \cite[\S 4.3, \S4.4]{thesis} we describe how the algebra map extends via completion (over $b_j$) to a \emph{dilating skew-endomorphism} $\phi^* : \k((x))(y) \to  \k((x))(y)$, meaning that $\abs{\phi_1(a)} = \abs a^n\ \forall a \in \k((x))$. Here, $n$ is the first integer with $\lambda_n \ne 0$ and we call $\q = \frac 1n$ the \emph{scale factor} of $\phi^*$. The induced non-Archimedean skew product $\phi_* : \P^1_\an \to \P^1_\an$ can be defined on a seminorm $\zeta \in \A^1_\an$ by $\norm[\phi_*(\zeta)]{f} = \norm[\zeta]{\phi^*(f)}^\q$. We say $\phi_*$ is \emph{simple} iff $n = 1$, equivalently $\q = 1$ or $\lambda_1 \ne 0$.

If $b=c$ we obtain a dynamical system representing the dynamics on a fibre fixed by $\phi$. In general, we may chain these semi-conjugacies together, considering an orbit $b_1 \mapsto b_2 \mapsto \cdots \mapsto b_N$ and a global model $Y$ of $X$ over $(b_j)_{j=1}^N$. It is most interesting to consider this situation where this orbit is a periodic cycle.

 \[
\begin{tikzcd}
\cdots \arrow{r} & \P^1_{\an, (j-1)} \arrow[two heads]{d}{\redct_{j-1}} \arrow{r}{\phi_*^{(j-1)}} & \P^1_{\an, (j)} \arrow[two heads]{d}{\redct_j} \arrow{r}{\phi_*^{(j)}} & \P^1_{\an, (j+1)} \arrow[two heads]{d}{\redct_{j+1}} \arrow{r}{\phi_*^{(j+1)}}& \P^1_{\an, (j+2)} \arrow[two heads]{d}{\redct_{j+2}} \arrow{r} & \cdots \\
 \cdots \arrow[swap]{r} & X_{b_{j-1}} \arrow[swap]{r}{\phi} & X_{b_j} \arrow[swap]{r}{\phi} & X_{b_{j+1}} \arrow[swap]{r}{\phi} & X_{b_{j+2}} \arrow{r} & \cdots
\end{tikzcd}
\]

\begin{defn}
 A \emph{chain of skew products} \[\left(\phi_*^{(j)} : \P^1_{\an, (j)}\right)_{j=1}^N,\] is a sequence of $1 \le N \le \infty$ copies of the Berkovich projective line $\P^1_{\an, (j)} = \P^1_\an(\hk)$ for $1 \le j \le N$ and skew products $\phi_*^{(j)} : \P^1_{\an, (j)} \to \P^1_{\an, (j+1)}$ for $1 \le j < N$. We say it is \emph{preperiodic} iff additionally for some $p \ge 1, n_0 \ge 0$ we have $N = p + n_0$ and $\phi_*^{(N)} : \P^1_{\an, (N)} \to \P^1_{\an, (n_0 + 1)}$; we call it \emph{periodic} when $n_0 = 0$. In this case, we extend the chain taking all indices eventually modulo $p$, e.g.\ $\P^1_{\an, (j)} = \P^1_{\an, (j+p)}$ for $j > n_0$. To minimise notation, we will write $\phi_*^n : \P^1_{\an, (j)} \to \P^1_{\an, (j+n)}$ for the composition \[\phi_*^{(j+n-1)} \circ \cdots \circ \phi_*^{(j+1)} \circ \phi_*^{(j)} : \P^1_{\an, (j)} \to \P^1_{\an, (j+n)}.\] 
\end{defn}

In this case, $\phi_1^{(j)}$ is given by \[\phi_1^{(j)}(x) = \lambda_{j, 1}x + \lambda_{j, 2}x^2 + \lambda_{j, 3}x^3 + \cdots \] where $\lambda_{j, n} \in \k$. If the $b_j$ are not critical, the induced (Berkovich) $\k$-rational skew products over these fibres are \emph{simple}; see \cite[\S 3.4, Proposition 4.9]{thesis}. Furthermore, the scale factor for $\phi_*^p$ in a $p$-periodic chain of skew products would be the product of individual scale factors.
 
 \begin{rmk}\label{rmk:skewstab:extendedfatoujulia}
Given a preperiodic chain as described above, for every $j > n_0$, $\phi_*^p : \P^1_{\an, (j)} \to \P^1_{\an, (j)}$ is the same kind of (autonomous) skew product studied in {\cite[\S3]{thesis}}. Therefore the Fatou and Julia sets are defined, and we may write them as $\F_{\an, (j)}$ and $\J_{\an, (j)}$ respectively. By a proof similar to the usual one for invariance of Fatou and Julia sets, one can show that \[\phi_*^{-1}(\F_{\an, (j)}) = \F_{\an, (j-1)},\ \phi_*(\F_{\an, (j)}) = \F_{\an, (j+1)},\ \phi_*^{-1}(\J_{\an, (j)}) = \J_{\an, (j-1)},\ \phi_*(\J_{\an, (j)}) = \J_{\an, (j+1)}.\] Naturally, we also find that Fatou components map to Fatou components. By defining attracting, indifferent and wandering components by their behaviour for $\phi_*^p$, we find that a Fatou component $U \subset \P^1_{\an, (j)}$ of one class maps to the component $\phi_*^{(j)}(U)  \subset \P^1_{\an, (j+1)}$ of the same class. The concepts of Julia point and Fatou component will be used later in the proof of \autoref{thm:extendvertexset}.
\end{rmk}

One feature that makes \autoref{thm:intro:stabskew} possible is the fact that a simple skew product $\phi_*$ does not increase the multiplicity of points. Failure of this property is responsible for the explosion of heights in parameters seen in \autoref{thm:intro:skewcounter}.

\begin{prop}[{\cite[Proposition 4.35, Corollary 4.41]{thesis}}]\label{prop:skewstab:multiplicitydivides}
 Let $\phi_*$ be a simple skew product over $\k((x))$, and $\zeta \in \P^1_\an$. Then $\mm(\phi_*(\zeta)) \mid \mm(\zeta)$, $\mg(\phi_*(\zeta)) \divides \mg(\zeta)$ and $\phi_*(\T_n) \subseteq \T_n$.
\end{prop}
 
 \subsection{Analytic Stability}

The following definitions are due to DeMarco and Faber \cite{DeF2}; they will translate algebraic stability from rational skew products over to skew products on the Berkovich projective line. 
  Recall that a vertex set $\Gamma \subset \P^1_\an(\hk)$ separates $\P^1_\an\sm \Gamma$ into the disjoint union of a collection $\mathcal{S}(\Gamma)$ of connected open affinoids called $\Gamma$-domains. When a $\Gamma$-domain has one boundary point we call it a $\Gamma$-disk, and if it has two, we call it a $\Gamma$-annulus.
  
\begin{defn}
Let $\left(\phi_*^{(j)},  \P^1_{\an, (j)}\right)_{j=1}^N$ be a (possibly periodic) chain of $N$ skew products, and let $\Gamma_{(j)} \subset \P^1_{\an, (j)}$ be vertex sets. 
\begin{itemize}
\item A $\Gamma_{(j)}$-domain $U$ will be called an \emph{F-domain} if $\phi_*^n(U) \cap \Gamma_{(j+n)} = \emp$ for all $n \ge 1$, and otherwise $U$ will be called a \emph{J-domain}. If $U$ is a $\Gamma_{(j)}$-disk, then it will be called an \emph{F-disk} or a \emph{J-disk}, respectively.
\item Write $\mathcal{J}(\Gamma_{(j)}) \subset \mathcal{S}(\Gamma_{(j)})$ for the subset consisting of all J-domains.
 \item Write $\mathcal{F}(\Gamma_{(j)}) \subset \mathcal{S}(\Gamma_{(j)})$ for the subset consisting of all F-domains.
 \item We say $\zeta \in \Gamma_{(j)}$ is \emph{destabilising} iff $\exists n \in \N$ such that $\phi_*^n(\zeta) \in U \in \mathcal{J}(\Gamma_{(j+n)})$.
 \item We say that $\left(\phi_*^{(j)}, \Gamma_{(j)}\right)_{j=1}^N$ is \emph{analytically stable} iff each of the $\Gamma_{(j)}$ have no destabilising points.
 \end{itemize}
\end{defn}

Caution that in \cite{DeF2}, the set $\mathcal{J}(\Gamma)$ is defined to also include the elements of $\Gamma$.
 We prefer the partition $\mathcal{S}(\Gamma) = \mathcal{J}(\Gamma) \cupdot \mathcal{F}(\Gamma)$ of $\Gamma$-domains. The next proposition provides an equivalent but apparently easier condition to satisfy. 
 
\begin{prop}\label{prop:stabskew:analyticstability}
Let $\left(\phi_*^{(j)},  \P^1_{\an, (j)}\right)_{j=1}^N$ be a periodic chain of $N$ skew products, and let $\Gamma_{(j)} \subset \P^1_{\an, (j)}$ be vertex sets. Then $\left(\phi_*^{(j)}, \Gamma_{(j)}\right)_{j=1}^N$ is analytically stable if and only if for every $\zeta \in \Gamma_{(j)}$, either $\phi_*(\zeta) \in \Gamma_{(j+1)}$ or $\phi_*(\zeta) \in U$, where $U \in \mathcal{F}(\Gamma_{(j+1)})$ is an F-domain.
\end{prop}

\begin{proof}
 If the latter condition fails for $\zeta \in \Gamma_{(j)}$ then $\phi_*(\zeta) \in U$ lies in a J-domain $U \in \mathcal{J}(\Gamma_{(j+1)})$; clearly this is destabilising. 
Conversely, suppose that $\zeta \in \Gamma_{(j)}$ is destabilising. Replace $\zeta$ with the last iterate $\phi_*^{n_0}(\zeta)$ contained in $\Gamma_{(j+n_0)}$; this way we may also assume that $\zeta \in \Gamma_{(j)}$ but $\phi_*(\zeta) \nin \Gamma_{(j+1)}$. Now, for some $n \ge 1$, we have $\phi_*^n(\zeta) \in V$ where $V \in \mathcal{J}(\Gamma_{(j+n)})$ is a J-domain. If this is true for $n=1$, then we are done. Otherwise $\phi_*(\zeta) \in U$, where $U \in \mathcal{F}(\Gamma_{(j+1)})$ is an F-domain. We know that $\phi_*^m(V) \cap \Gamma_{(j+n+m)} \ne \emp$ for some $m \in \N$. Hence $\phi_*^{n-1}(U)$ cannot contain $V$, else $U$ would also be a J-domain. Along a path from $\phi_*^n(\zeta) \in \phi_*^{n-1}(U) \cap V$ to some point in $V \sm \phi_*^{n-1}(U)$ we can find $\zeta_n \in \partial(\phi_*^{n-1}(U)) \cap V$. Because $\partial(\phi_*^{n-1}(U)) \subset \phi_*^{n-1}(\partial U)$ there also exists $\zeta_1 \in \partial U$ with $\phi_*^{n-1}(\zeta_1) = \zeta_n$. Now, since $U$ is a $\Gamma_{(j+1)}$-domain, $\zeta_1 \in \Gamma_{(j+1)}$, and this is destabilising because $\zeta_n \in V$. Continuing this way, the proof concludes by induction on $n$.
\end{proof}

Now suppose $b \in B$ is fixed by $\phi_1$ and consider the induced non-Archimedean skew product $\phi_* : \P^1_\an \to \P^1_\an$ over $b$. Let $\Gamma = \Gamma(X_b)$ be the vertex set corresponding to this fixed fibre. Recall that a \emph{destabilising orbit} is an orbit $p, \phi(p), \dots, \phi^{n-1}(p) = p'$ such that $\phi(p')$ is an \emph{inverse destabilising curve} $D$, and $\phi^{-1}(p)$ is a \emph{destabilising curve} $C$. By \autoref{thm:intro:equivalence}, $\phi$ is algebraically stable if and only if $\phi$ has no destabilising orbits. Through the reduction map $\redct_b : \P^1_\an \to X_b$, we see that $C = \redct_b(\set{\zeta_1, \dots, \zeta_s})$ and $D = \redct_b(\set{\xi_1, \dots, \xi_t})$ correspond to finite subsets of Type II points (depending on reducibility). Further, $p = \redct_b(U)$ and $p' = \redct_b(V)$ for two $\Gamma$-domains $U, V$ (one should also consider their Galois conjugates, or more simply consider the unique $\Gamma_G$-domains $U_G, V_G$ in the valuative tree $\V$). The proper transform of $C$ is $p$, meaning $\phi_*(\zeta_j) \in U$ for every $j$; similarly, $\phi(p') = D$ implies $\phi_*(V)$ contains $\set{\xi_1, \dots, \xi_t}$. Also, $\phi_*^{n-1}(U) = V$ because $\phi^{n-1}(p) = p'$. Clearly $V$ is a J-domain, therefore $\zeta_1, \dots, \zeta_s$ are destabilising points. One can check that in this way destabilising Type II points always give rise to destabilising orbits on $X_b$. After generalising this to periodic fibres, we have proven the following.

\begin{prop}\label{prop:stabskew:analyticallystable}
 Let $h : X \to B$ be a birationally ruled surface and $\phi : X \dashto X$ be a rational skew product over $B$. Suppose that $b_1, b_2, \dots b_N$ form a cycle for $\phi_1$ and let $\left(\phi_*^{(j)},  \P^1_{\an, (j)}\right)_{j=1}^N$ be the induced periodic chain of $N$ non-Archimedean skew products. Then $\phi$ has no destabilising orbit contained in the fibres $X_{b_1}, X_{b_2}, \dots X_{b_N}$ if and only if $\left(\phi_*^{(j)}, \Gamma_{(j)}\right)_{j=1}^N$ is analytically stable. In particular, a Type II destabilising point $\zeta \in \Gamma_{(j)}$ corresponds to an irreducible destabilising curve $E \subset X_{b_j}$ for $\phi$.
\end{prop}

\thesisarticle{\subsection{Finding an Analytically Stable and Smooth Vertex Set}\label{sec:thetrickybit}}{
\section{Finding an Analytically Stable and Smooth Vertex Set}\label{sec:thetrickybit}}

Through the ideas in the last section, we have reduced \autoref{thm:skewstab:periodicgeom} to the following theorem.

\begin{thm}\label{thm:extendvertexset}
 Let $\left(\phi_*^{(j)} : \P^1_{\an, (j)}\right)_{j=1}^N$ be a periodic chain of simple skew products over $\K$. Given any finite sets $\Gamma_{(j)} \subset \GV_m \subset \P^1_{\an, (j)}$ of Type II points we can find finite supersets $\Gamma'_{(j)} \supseteq \Gamma_{(j)}$ in each $\P^1_{\an, (j)}$ which are smooth and analytically stable in the chain.
\end{thm}

The idea of the proof is to alternately extend our vertex set in two different ways. We start with the vertex set $\Gamma_0 = \Gamma$ and should initially take a smooth $m$-convex hull $\leadsto \tilde\Gamma_0$.
To achieve stability, one can add points in the forward orbit of $\tilde\Gamma_0$; using the dynamical properties of $\phi_*$, we show that finitely many will do. However, it is likely that the new vertex set, $\Gamma_1$, represents a singular surface. Now, we could extend the vertex set to be smooth again, but then one should expect that we introduced new points to the smooth vertex set $\tilde\Gamma_1$, which would destroy analytic stability. One could imagine having no choice but to repeat these two steps ad infinitum. 
\[\Gamma \subseteq \tilde\Gamma_0 \subset \Gamma_1 \subset \tilde\Gamma_1 \subset \Gamma_2 \subset \tilde\Gamma_2 \subset \Gamma_3 \subset \cdots\]
This would be useless because adding infinitely many points to $\Gamma$ would correspond to blowing up $X$ infinitely many times, which is absurd. We need a \emph{finite} vertex set. In spite of this concern, our procedure alternates between adding points to restore smoothness and strategically including points in the forward orbit to restore analytic stability. 

For the process to terminate at $\tilde \Gamma_n = \Gamma_{n+1}$, we need this to be analytically stable. According to \autoref{prop:stabskew:analyticstability}, we want for every $\zeta \in \tilde \Gamma_n$ to either find $\phi_*(\zeta)$ in $\tilde \Gamma_n$ or in one of its F-domains; we cannot have $\phi_*(\zeta)$ in a J-domain.

There are two major problems to overcome, although several others appear in the proof. First, as we add points to the vertex sets, either for smoothness or stability, there is a risk we might destroy what were F-domains, or turn them into J-domains. This could turn some of our old `stable' vertices into destabilising points. Second, the algorithm could conceivably perpetuate if at every stabilisation step there is some $\zeta \in \tilde \Gamma_n \sm \Gamma_n$ with $\phi_*(\zeta) \in V$, a J-disk for $\tilde \Gamma_n$. This can only be caused by $\phi_*$ folding the subtree $\Hull(\Gamma_n)$ near $\zeta$. The purpose of the `No-Folding Lemma' below, is to add so many ramified points to $\Gamma_0$ that this behaviour is completely controlled. To solve the first problem, we design an intricate set of rules for the stabilisation step $\tilde \Gamma_n \leadsto \Gamma_{n+1}$ with termination of the sequence in mind. More specifically, in each stabilisation step we earmark areas of $\P^1_\an$ as `persistent F-domains' where no more points should be added to future vertex sets, and enforce such rules as we add sequences of vertices to future vertex sets. We also take care to reduce the amount of smoothing that will be required by any actions we take. Our procedure combines the dynamical (Fatou-Julia) theory of $\phi_*$ and the geometric multiplicity structure of $\P^1_\an$ discussed above. 

\begin{lem}\label{lem:indiffwand}
 Let $\phi_* : \P^1_\an \to \P^1_\an$ be a simple skew product and $U$ be a periodic indifferent component. Then the set of (pre)periodic points inside $U$ is connected. Moreover, every point, $\zeta \in U$ which is not (pre)periodic, lies in a `wandering disk': $\zeta \in D \subseteq U$ such that the disks $\phi_*^n(D)$ are distinct, and $\xi = \partial D$ is periodic.
\end{lem}

\begin{proof}
  Since $U \subset \Inj(\phi)$, every preperiodic point in $U$ is periodic. Recall that the boundary points of $U$ are all periodic by {\cite[Theorem 3.90]{thesis}}. The connectedness of the periodic points is a simple application of {\cite[Corollary 3.59]{thesis}}, which says that intervals (in $U$ say) map homeomorphically and isometrically, because $\q = 1$ and $U \subset \Inj(\phi)$. If $\alpha, \beta$ are fixed by $\phi_*^n$ for some $n \in \N$, then $[\alpha, \beta]$ must map identically to itself under $\phi_*^n$, and so it is an interval of periodic points. Let $\zeta$ be a non-periodic in $U$. By the connectedness of periodic points, the periodic points of $U$ cannot lie in two directions at $\zeta$. Let $\xi$ be the closest periodic point to $\zeta$ in $\Hull(\partial U)$. Now consider the disk $D = \vec v(\zeta) \in \Dir\xi$. Note that $\phi_*^n(D)$ must remain disjoint from the (pre)periodic points for every $n$. If $D$ is wandering, i.e.\ the disks $\phi_*^n(D)$ are distinct, then we are done. Otherwise, since $\xi$ is periodic, $\xi \nin \phi_*^n(D)$, and $\phi_*^n(\xi) \nin D$, we have $\phi_*^n(D) = D$. It follows that $(\xi, \zeta) \cap (\xi, \phi_*^n(\zeta)) = (\xi, \zeta') \ne \emp$ for some $\zeta' \in D$. Since $\phi_*$ is isometric of $U$, it is the identity on $[\xi, \zeta']$, and so $\zeta'$ is $n$-periodic $\contra$.
\end{proof}

We can say that an open affinoid $V$ is in the \emph{attracting basin} of $\xi$ iff for some $n$, $\phi_*^n(V) = D$ is a disk, with $\xi, D$ as in the proposition. This matches the terminology for Fatou wandering components.

\begin{lem}[No-Folding Lemma]\label{lem:nofolding}
 Let $\phi_*$ be a simple skew product. Then there exists a finite subtree $T \subset \Hull(\Ram(\phi)) \subset \P^1_\an$ with finite hyperbolic diameter (to be precise, the convex hull of finitely many Type II points), with the following property. Given any interval $I = (\alpha, \beta) \subset \P^1_\an \sm T$ such that $I \subseteq \Inj(\phi)$ or the direction $\vec v(\beta)$ at $\alpha$ is disjoint from $T$, then $\phi_*$ maps $I$ injectively.
\end{lem}

Intuitively, the last condition says that $I$ is a segment of a path that hits $T$. We choose $T$ to be a large enough piece of $\Hull(\Crit(\phi)) \cap \bH$ such that $\phi_*$ restricted to $\Ram(\phi) \sm T$ is injective. This works because near a critical point of multiplicity $n$, the reduction $\overline \phi$ at Type II points is $z \mapsto z^d$, where the directions of $\Ram(\phi) \sm T$ correspond to $0$ and $\infty$. 

\begin{proof}
 If $I \subseteq \Inj(\phi)$ then injectivity is always guaranteed by {\cite[Corollary 3.59]{thesis}}. Let $a \in \Crit(\phi)$, then {\cite[Proposition 3.60, Theorem 3.61]{thesis} say} that there is an $\eps > 0$ such that $\Ram(\phi) \cap \CD_\an(a, \eps) = [a, \zeta(a, \eps)]$ and this interval is mapped homeomorphically by $\phi_*$.
 
 For each $a \in \Crit(\phi)$ pick such a point $\zeta(a, \eps)$ with $\eps \in |K|$ and let these be the endpoints of $T$, which is the convex hull of the endpoints. Clearly $T$ has finite hyperbolic diameter because these endpoints are not Type I and the maximum length path must be between two of the finitely many endpoints.
 
 Now let $I = (\alpha, \beta)$ be an interval that intersects $\Ram(\phi)$ but the direction $\vec v(\beta)$ at $\alpha$ is disjoint from $T$. Observe that by construction of $T$, $\Hull(\Ram(\phi)) \sm T$ is a disjoint union of intervals of the form $[a, \zeta(a, \eps))$ where $a \in \Crit(\phi)$. 
 By unique path-connectedness of $\P^1_\an$, we know that $I$ intersects only one such interval. Suppose we have $0 \le \delta < \eps$ minimal such that $\zeta(a, \delta) \in [\alpha, \beta]$.
 Suppose $\alpha \nin [a, \zeta(a, \eps)]$, then $\vec v(\beta) = \vec v(\zeta(a, \delta)) = \vec v(T)$, a contradiction. 
 By unique path connectedness, we can see that $(\beta, \alpha) \cap [a, \alpha) = [\zeta(a, \delta), \alpha)$ and by definition that $(\beta, \zeta(a, \delta)) \cap [a, \zeta(a, \eps)) = \emp$. Therefore each piece of \[(\beta, \alpha) =  (\beta, \zeta(a, \delta)) \cupdot [\zeta(a, \delta), \alpha)\] is mapped homeomorphically by $\phi_*$. Moreover, the conclusion of {\cite[Theorem 3.61]{thesis}} was that \[D_\an(a, \eps) \cap \phi_*^{-1}[\phi_*(a) , \phi_*(\zeta(a, \eps))] = [a , \zeta(a, \eps)]\] so our two subintervals of $(\alpha, \beta)$ have disjoint images under $\phi_*$. This completes the proof.
\end{proof}

\begin{proof}[Proof of \autoref{thm:extendvertexset}]
To give an exposition uncluttered by indices, we will reduce to the period $N=1$ case and remove all $(j)$ indices from the notation and hypotheses. At the end of the proof we make various remarks about the general case. 

Let $\Gamma \subset \P^1_\an$ and let $m_0 = \max_{\zeta \in \Gamma} \mg(\zeta)$ be its maximum generic multiplicity; hence $\Gamma \subset \GV_{m_0}$. 
For the purposes of this proof, given any finite set $\Sigma \subset \GV_{m_0}$, we denote by $\tilde \Sigma$ the smooth $m_0$-convex hull of $\Sigma$. \autoref{prop:skewstab:smoothconvex} says that any such smooth $m_0$-convex hull is  smooth. If we have already added leaves to $\Hull(\Sigma)$ as in \autoref{prop:skewstab:smoothhull} to make the points of $\Sigma$ \flanked{}, this smooth convex hull is the same as taking $\Hull(\Sigma) \cap \GV_{m_0}$. \autoref{prop:skewstab:multsubtreevalency} and \autoref{rmk:skewstab:specialgeneric} say there are at most two such `special' directions at $\zeta$ (one is $\vec v(\infty)$) possibly needing a point added to $\Sigma$. Furthermore, assuming $\Sigma$ contains a point of generic multiplicity $1$, any \unflanked{} $\zeta \in \Sigma$ is satellite and only one (of the two) special direction $\bvec v$ at $\zeta$ requires an additional vertex to make $\zeta$ \flanked{}; see \autoref{prop:skewstab:smoothhull}.
One should view this Type II point $\zeta$ of multiplicity $\mm(\zeta) = m$ as part of an edge in $\T_m$, and the additional vertex required in $\bvec v$ will be one of the vertices $\xi_1, \xi_2$ bounding this edge, with $\mm(\bvec v) = \mg(\xi_i) = m < \mg(\zeta)$.

At first let $T$ be the set from the No-Folding Lemma \ref{lem:nofolding} and let $\Gamma_0 = \Gamma \cup (T \cap \GV_{m_0})$. Then $\tilde \Gamma_0$ is the smallest smoothly $m_0$-convex vertex set containing $\Gamma$ and whose convex hull contains $T$. We proceed to construct an increasing sequence of vertex sets $\Gamma_n$, $n \in \N$, each obtained by adding points to $\tilde \Gamma_{n-1}$ according to a procedure we will describe shortly. Note that $\tilde\Gamma_0$ is smooth and so it already contains a point of generic multiplicity $1$.

\[\Gamma \subseteq \tilde\Gamma_0 \subseteq \Gamma_1 \subseteq \tilde\Gamma_1 \subseteq \Gamma_2 \subseteq \tilde\Gamma_2 \subseteq \Gamma_3 \subseteq \cdots\]

In step $n$ of the recursion, we may designate any Berkovich open disk $D$ as a \emph{persistent F-disk}. Let $\mathcal{D}_n$ denote the collection of all such disks. A persistent F-disk $D \in \mathcal{D}_n$ will obey the following axioms, which we will prove are conserved in all future steps.

\textbf{Persistent F-disk Axioms}
\begin{enumerate}[label = (\roman*), ref = Ax.(\roman*)]
 \item $\partial D \in \tilde \Gamma_n$; \label{ax:bdry}
 \item $\mm(D) = \mg(\partial D)$ (the direction $D$ is generic at $\partial D$); \label{ax:gen}
 \item $D \cap \Gamma_n = \emp$;\label{ax:fdom}
 \item $\phi_*(D) \subseteq D' \in \mathcal{D}_n$; and \label{ax:iter}
 \item $D \in \mathcal{D}_m$ for every $m \ge n$. \label{ax:pers}
\end{enumerate}
First, note that the condition that $D$ is a generic direction means that for any set $\Sigma$ disjoint from $D$, the smooth convex hull $\tilde \Sigma$ is also disjoint from $D$. This is because \autoref{prop:skewstab:smoothhull} shows that $\tilde \Sigma = \Hull(\Lambda) \cap \GV_{m_0}$ where $\Lambda \sm \Sigma$ only has points in special directions from vertices of $\Sigma$. Hence $D \cap \Gamma_n = \emp \implies D \cap \tilde\Gamma_n = \emp$.
Second, these axioms mean that $D \in \mathcal{D}_n$ will be an F-disk for $\tilde\Gamma_n$ and all future vertex sets $\tilde \Gamma_m$ and $\Gamma_{m+1}$ for every $m \ge n$. Indeed, if $D$ is disjoint from $\tilde \Gamma_n$ and $\partial D \in \tilde\Gamma_n$, then $D$ is a $\tilde\Gamma_n$-disk by definition; given $\phi_*(D) \subseteq D' \in \mathcal{D}_n$, $\phi_*(D') \subseteq D'' \in \mathcal{D}_n$ and so on, we can see $D$ is an F-domain; the rest follows by inclusions. Third, the union of the persistent F-disks at step $n$ is a forward invariant open subset of $\P^1_\an$.

\begin{clm}\label{clm:main}
 For each $\zeta \in \Gamma_n$, we will have that $\mm(\zeta) \le m_0$ and either $\phi_*(\zeta) \in \Gamma_n$ or $\phi_*(\zeta) \in D$ for some $D \in \mathcal{D}_n$.
\end{clm}

Now we lay out the recursion for building the $\Gamma_{n+1}$ and explain why \autoref{clm:main} will hold for $n+1$. The first part of the claim will follow from \autoref{prop:skewstab:multiplicitydivides} given that we only add images of existing points of $\Gamma_{n+1}$. Initially let $\mathcal{D}_{n+1}$ be $\mathcal{D}_n$, and $\Gamma_{n+1}$ be $\tilde \Gamma_n$. For every $\zeta \in \tilde \Gamma_n$, we apply the following rules.

\begin{enumerate}[label = (\roman*), ref = (\roman*)]
 \item \label{item:evf} Suppose the iterates of $\zeta$ eventually hit a persistent F-disk, then let $N \in \N$ be the first integer such that $\phi_*^N(\zeta) \in D \in \mathcal{D}_{n+1}$. We choose to include $\phi_*^j(\zeta)$ in $\Gamma_{n+1}$ for every $0 < j < N$. These points obey the claim, with the last having its image in $D$. Practically, this rule ensures that we never add points to $\Gamma_{n+1}$ from our persistent F-disks.
 
  \item \label{item:vtx} If the iterates of $\zeta$ eventually hit another point $\xi = \phi_*^N(\zeta) \in \Gamma_{n+1}$, then we choose to include $\phi_*^j(\zeta)$ in $\Gamma_{n+1}$ for every $0 < j < N$. These points satisfy the claim.
  
   \item \label{item:preper} Suppose $\zeta$ is preperiodic. Then include all of $\Orb^+(\zeta)$ in $\Gamma_{n+1}$. Each of these points obey the claim since their images lie in $\Gamma_{n+1}$. Note that every Type II Julia point is preperiodic by the corollary to Benedetto's `no wandering domains' theorem, which is generalised to simple skew products defined over $\k((x))$ {\cite[Corollary 3.100]{thesis}}. So in other cases below, $\zeta$ must be a Fatou point.
 
 \item \label{item:attr} If $\zeta$ lies in an attracting basin $U$, let $\gamma_1, \dots, \gamma_p \in \P^1$ be the attracting cycle and observe that their multiplicities must be the same, say $m$. The iterates of $\zeta$ must converge to this cycle. Since we are not in case (i), we may assume that the $\gamma_j$ are not contained in F-disks of $\mathcal{D}_{n+1}$; hence we will create them now. Pick $M$ large enough and $t_j \in \N$ large enough, such that for every $1 \le j \le p$, 
\begin{align*}
 \phi_*^{M+j}(\zeta) &\in D_\an\left(\gamma_j,\ |x|^\frac {t_j-1}{m}\right) \sm D_\an\left(\gamma_j,\ |x|^\frac {t_j}{m}\right),\\
 \phi_*^{M+p+1}(\zeta) &\in D_\an\left(\gamma_1, |x|^\frac {t_1}{m}\right),\\
 D_\an\left(\gamma_j, |x|^\frac {t_j}{m}\right) &\subset U \sm \Gamma_{n+1},\\
 \text{and \quad } \mm\left(D_\an\left(\gamma_j, |x|^\frac {t_j}{m}\right)\right) &= m.
\end{align*}
 Now include $\phi_*^r(\zeta)$ into $\Gamma_{n+1}$ for every $r \le M+p$, and the $D_j = D_\an(\gamma_j, |x|^\frac {t_j}{m})$ into $\mathcal{D}_{n+1}$. Note that $\phi_*(D_j) = D_{j+1}$ for $1 \le j < p$ and $\phi_*(D_p) \subset D_1$. One can check that $\mm(\partial D_j) = \mg(\partial D_j) = m$ and so $\zeta(\gamma_j, |x|^\frac {t_j}{m}) = \partial D_j$ is a free point. We must show that this point will be needed to form a smooth $m_0$-convex set with $\phi_*^{M+j}(\zeta)$ and $\Gamma_{n+1}$. 
 If $\partial D_j$ is on the path between $\phi_*^{M+j}(\zeta)$ and $\Gamma_n$, then all such free points of multiplicity at most $m_0$ will be included in $\tilde\Gamma_{n+1}$, thus guaranteeing $\partial D_j \in \tilde\Gamma_{n+1}$. Otherwise the join $\xi_j = \phi_*^{M+j}(\zeta) \wedge \partial D_j$ of these two points is in the annulus written above; see \autoref{fig:attrbasin}. Hence \[\xi_j \in \left(\zeta\left(\gamma_j,\ |x|^\frac {t_j-1}{m}\right), \zeta\left(\gamma_j,\ |x|^\frac {t_j}{m}\right)\right),\] so $\xi_j$ is satellite with $\mm(\xi_j) = m$ and $m < \mg(\xi_j) \le m_0$. 
  By the discussion at the start of the proof about smooth convex hulls, $\tilde\Gamma_{n+1}$ must also incorporate the nearest free point in the special direction $\vec v(\gamma_j)$ at $\xi_j$, namely $\partial D_j$. We have verified the axioms for these new persistent F-disks $D_j \in \mathcal{D}_{n+1}$. Since $\phi_*^{M+p+1}(\zeta) \in D_1$, the claim is satisfied in this case.
    \begin{figure}
  \centering
   \resizebox{\textwidth}{!}{\input{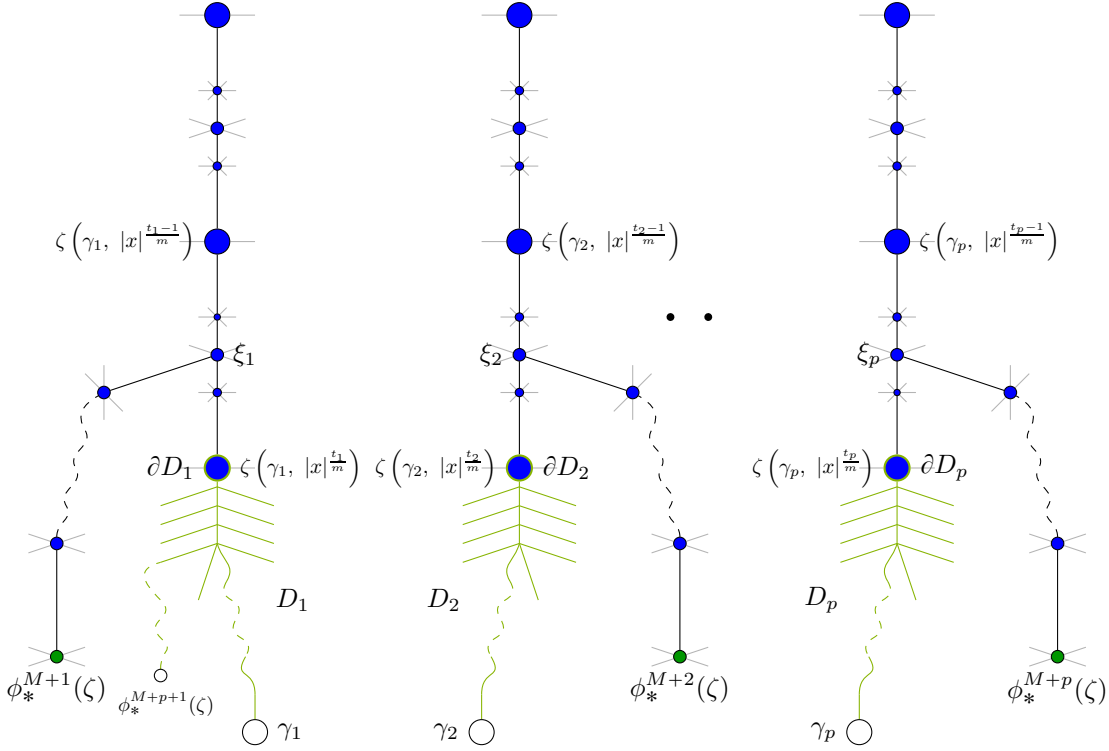}}
  \caption{An attracting basin, shown with coloured vertices of {\gamcol$\Gamma_{n+1}$}, {\smgamcol$\tilde\Gamma_{n+1}$} and persistent F-disks of {\gamdomcol$\mathcal D_{n+1}$}.}
  \label{fig:attrbasin}
\end{figure}%
 
 \item \label{item:wand} If $\zeta$ is in a wandering Fatou component $U$, by the `no wandering domains theorem' {\cite[Theorem 3.99]{thesis}} $\zeta$ is in the attracting basin of some cycle $\xi_1, \dots, \xi_p$. To be precise, there is a minimal $N \ge 0$ such that for each $0 \le j \le p$, $\phi_*^{N+ j}(U)$ is a residue class at $\xi_j$. Furthermore, we can increase $N$ such that for every $t > N$, $\phi_*^t(U)$ is a generic direction and is disjoint from $\Gamma_{n+1} = \emp$; by not being in case (i), we also know these are not disks in $\mathcal{D}_n$. Now we choose to include $\phi_*^t(\zeta)$ in $\Gamma_{n+1}$ for every $t \le N+p$, and $\phi_*^t(U)$ into $\mathcal{D}_{n+1}$ for every $t > N+p$. The $\xi_j$ have multiplicity at most $m_0$ because they each have infinitely many directions containing points of multiplicity at most $m_0$. Now observe that by including $\phi_*^{N+1}(\zeta), \dots, \phi_*^{N+p}(\zeta)$ we ensure that each $\xi_j$ lies on a closed interval between $\phi_*^{N+j}(\zeta)$ and $\Gamma_n$. Therefore we will find $\xi_j = \partial(\phi_*^{N+j+pl}(U)) \in \tilde \Gamma_{n+1}$ for every $l \ge 0$ and $1 \le j \le p$. This shows that $\phi_*^t(U) \in \mathcal{D}_{n+1}$ are good persistent F-disks for $t > N+p$. Since $\phi_*(\phi_*^{N+p}(\zeta)) \in \phi_*^{N+p+1}(U)$, one can see that we have satisfied the claim with $\phi_*^t(\zeta)$ for every $0 \le t \le N+p$.
 
\item \label{item:indiffwand} If $\zeta$ is not preperiodic but eventually in the indifference domain, then by \autoref{lem:indiffwand} $\exists N \in \N$ such that $\phi_*^N(\zeta)$ lies in the `attracting basin' of a cycle $\xi_1, \dots, \xi_p$. We proceed as in the previous case.
\end{enumerate}

This finishes the description of the recursion; we have verified \autoref{clm:main} and upheld the axioms of the persistent F-disks.

Now suppose that the sequence $(\Gamma_n)$ is eventually constant with $\Gamma_{n+1} = \tilde \Gamma_n$ for some $n \in \N$. Then $\Gamma_{n+1} = \Gamma'$ must be an analytically stable (by \autoref{clm:main}), smooth (by \autoref{prop:skewstab:smoothconvex}), finite set of Type II points. We are done.
 
 Otherwise, we suppose for contradiction that each new set $\Gamma_n$ is strictly larger. In each smoothing and stabilisation step, the procedure adds points whose multiplicity does not exceed that of the points added in the previous step; see \autoref{prop:skewstab:smoothhull} in the smooth case and \autoref{prop:skewstab:multiplicitydivides} in the iterative case. The multiplicity of new points added in step $n$ is a deceasing sequence of positive integers, so it must eventually be constant. Hence, let $m_+ \le m$ be the largest multiplicity appearing in $\Gamma_{n+1} \sm \tilde \Gamma_n$ for infinitely many $n$. 
 The rest of the proof forks into two major cases.
\begin{enumerate}[label = (\Alph*)]
\item For infinitely many $n$, there are \unflanked{} satellite points in $\Gamma_{n+1}$ of multiplicity $m_+$.
\item For large enough $n$, every multiplicity $m_+$ point in $\Gamma_n$ is \flanked{}, so $\tilde \Gamma_n$ is the $m_0$-convex hull of $\Gamma_n$.
\end{enumerate}

The following claim will support the remainder of the proof.

\begin{clm}\label{clm:zerodirn}
 Let $\zeta \in \GV_{m_0} \sm \tilde\Gamma_0$ and let $\alpha$ be the closest point of $\tilde\Gamma_0$ to $\zeta$; define $\bvec u = \vec v(\Gamma_0) = \vec v(\alpha)$ at $\zeta$. Suppose that $\phi_*(\zeta) \nin \tilde\Gamma_1$ and $\phi_*(\zeta) \nin D$ for any persistent F-disk $D \in \mathcal D_1$. Then $\phi_*$ maps $(\alpha, \beta)$ injectively to $(\phi_*(\alpha), \phi_*(\beta))$, and $\phi_\#(\bvec u) = \vec v(\Gamma_0) = \vec v(\phi_*(\alpha))$ at $\phi_*(\zeta)$. Moreover every $\bvec v \ne \bvec u$ at $\zeta$ is a good direction at $\zeta$ with $\phi_\#(\bvec v) \ne \vec v(\Gamma_0)$.
\end{clm}

\begin{figure}[H]
\centering
\begin{subfigure}{.4\textwidth}
  \centering
   \begin{tikzpicture}
	\begin{pgfonlayer}{nodelayer}
		\node [style=empty, label={above:$\phi_*(\beta)$}] (147) at (4.25, -1.25) {};
		\node [style=none] (148) at (3.5, -1.5) {};
		\node [style=black vertex, label={above:$\zeta$}] (2) at (2, 2) {};
		\node [style=none] (88) at (0.5, -2) {};
		\node [style=none] (11) at (-3, 2) {};
		\node [style=empty, label={below:$\phi_*(\zeta)$}] (19) at (2, -2) {};
		\node [style=none] (37) at (-1.5, 3.5) {};
		\node [style=none] (38) at (-2.5, 3.5) {};
		\node [style=none] (39) at (-2.5, 0.5) {};
		\node [style=none] (40) at (-1.5, 0.5) {};
		\node [style=none] (42) at (-1.5, -0.5) {};
		\node [style=none] (43) at (-2.5, -0.5) {};
		\node [style=none] (44) at (-2.5, -3.5) {};
		\node [style=none] (45) at (-1.5, -3.5) {};
		\node [style=none] (47) at (2.5, 3.5) {};
		\node [style=none] (48) at (1.5, 3.5) {};
		\node [style=none] (49) at (1.5, 0.5) {};
		\node [style=none] (50) at (2.5, 0.5) {};
		\node [style=none] (61) at (3.5, -2.5) {};
		\node [style=none] (64) at (3.5, 1.5) {};
		\node [style=empty, label={above:$\beta$}] (65) at (4.25, 2.75) {};
		\node [style=none] (82) at (2.5, -0.5) {};
		\node [style=none] (83) at (1.5, -0.5) {};
		\node [style=none] (84) at (1.5, -3.5) {};
		\node [style=none] (85) at (2.5, -3.5) {};
		\node [style=none] (90) at (-5.5, 0.75) {};
		\node [style=none] (93) at (-4, 1.875) {};
		\node [style=none] (94) at (-4.75, 1.5) {};
		\node [style=none] (99) at (-3, -2) {};
		\node [style=none] (100) at (-5.5, -0.75) {};
		\node [style=none] (103) at (-4, -1.875) {};
		\node [style=none] (104) at (-4.75, -1.5) {};
		\node [style={G{n}}, label={left:$\Gamma_0$}] (106) at (-6.25, 0) {};
		\node [style=none] (117) at (1, 2) {};
		\node [style=none] (118) at (0, 2) {};
		\node [style=none] (87) at (0.5, 2) {};
		\node [style=none] (136) at (-1, 2) {};
		\node [style={G{n}}, label={above:$\alpha$}] (137) at (-2, 2) {};
		\node [style=none] (141) at (-1, -2) {};
		\node [style=empty] (142) at (-2, -2) {};
		\node [style=none] (143) at (1, -2) {};
		\node [style=none] (144) at (0, -2) {};
		\node [style=none] (145) at (3.5, 2.5) {};
	\end{pgfonlayer}
	\begin{pgfonlayer}{edgelayer}
		\draw [style=grey edge] (19) to (61.center);
		\draw [style=grey edge] (48.center) to (2);
		\draw [style=grey edge] (2) to (50.center);
		\draw [style=grey edge] (47.center) to (2);
		\draw [style=grey edge] (2) to (49.center);
		\draw [style=grey edge] (2) to (64.center);
		\draw [style=grey edge] (65) to (2);
		\draw [style=grey edge] (83.center) to (19);
		\draw [style=grey edge] (19) to (85.center);
		\draw [style=grey edge] (82.center) to (19);
		\draw [style=grey edge] (19) to (84.center);
		\draw [style=mydashed] (137)
			 to [in=45, out=-180] (11.center)
			 to [in=45, out=-135, looseness=1.25] (93.center)
			 to [in=0, out=-135, looseness=1.50] (94.center)
			 to [in=0, out=-180, looseness=1.25] (90.center)
			 to [in=45, out=-180, looseness=1.50] (106);
		\draw [style=mydashed] (142)
			 to [in=-45, out=180] (99.center)
			 to [in=-45, out=135, looseness=1.25] (103.center)
			 to [in=0, out=135, looseness=1.50] (104.center)
			 to [in=0, out=180, looseness=1.25] (100.center)
			 to [in=-45, out=180, looseness=1.50] (106);
		\draw [style=thick arrow] (19) to node [midway, above] {$\phi_\#(\bvec u)$} (88.center);
		\draw [style=mydashed] (2)
			 to [in=45, out=-180, looseness=1.25] (117.center)
			 to [in=45, out=-135, looseness=1.50] (118.center)
			 to [in=45, out=-135, looseness=1.50] (136.center)
			 to [in=0, out=-135, looseness=1.25] (137);
		\draw [style=thick arrow] (2) to node [midway, above] {$\bvec u$} (87.center);
		\draw [style=mydashed] (142)
			 to [in=135, out=0, looseness=1.25] (141.center)
			 to [in=135, out=-45, looseness=1.50] (144.center)
			 to [in=135, out=-45, looseness=1.50] (143.center)
			 to [in=180, out=-45, looseness=1.25] (19);
		\draw [style=grey edge] (43.center) to (142);
		\draw [style=grey edge] (142) to (45.center);
		\draw [style=grey edge] (42.center) to (142);
		\draw [style=grey edge] (142) to (44.center);
		\draw [style=grey edge] (38.center) to (137);
		\draw [style=grey edge] (137) to (40.center);
		\draw [style=grey edge] (37.center) to (137);
		\draw [style=grey edge] (137) to (39.center);
		\draw [style=thick arrow] (2) to node [midway, below right] {$\bvec v$} (145.center);
		\draw [style=grey edge] (147) to (19);
		\draw [style=thick arrow] (19) to node [midway, below right] {$\phi_\#(\bvec v)$} (148.center);
	\end{pgfonlayer}
\end{tikzpicture}
  \caption{The claimed situation.}
  \label{fig:claimzerodirn:aim}
\end{subfigure}%
\hfill%
\begin{subfigure}{.58\textwidth}
  \centering
  \begin{tikzpicture}
	\begin{pgfonlayer}{nodelayer}
		\node [style=none] (115) at (0.5, 2) {};
		\node [style=none] (117) at (-1, 2) {};
		\node [style=none] (119) at (1, 2) {};
		\node [style=none] (120) at (0, 2) {};
		\node [style=none] (111) at (1, -2) {};
		\node [style=none] (113) at (3, -2) {};
		\node [style=none] (114) at (2, -2) {};
		\node [style=none] (105) at (4.25, -1.25) {};
		\node [style=none] (106) at (3.75, -1.25) {};
		\node [style=none] (107) at (3.75, -2.75) {};
		\node [style=none] (108) at (4.25, -2.75) {};
		\node [style=smooth n, label={below:$\phi_*(\zeta)$}] (15) at (0, -2) {};
		\node [style={G{n}}, label={above:$\alpha$}] (1) at (-2, 2) {};
		\node [style=black vertex, label={above:$\zeta$}] (2) at (2, 2) {};
		\node [style=none] (26) at (-2, -2) {};
		\node [style=none] (33) at (-1, -2) {};
		\node [style=none] (37) at (-1.5, 3.5) {};
		\node [style=none] (38) at (-2.5, 3.5) {};
		\node [style=none] (39) at (-2.5, 0.5) {};
		\node [style=none] (40) at (-1.5, 0.5) {};
		\node [style=none] (42) at (0.5, -0.5) {};
		\node [style=none] (43) at (-0.5, -0.5) {};
		\node [style=none] (44) at (-0.5, -3.5) {};
		\node [style=none] (45) at (0.5, -3.5) {};
		\node [style=none] (47) at (2.5, 3.5) {};
		\node [style=none] (48) at (1.5, 3.5) {};
		\node [style=none] (49) at (1.5, 0.5) {};
		\node [style=none] (50) at (2.5, 0.5) {};
		\node [style=none] (59) at (6.5, -3.5) {};
		\node [style=none] (60) at (6.5, -0.5) {};
		\node [style=none] (61) at (7.5, -2.5) {};
		\node [style=none] (64) at (3.5, 1.5) {};
		\node [style=none] (65) at (3.5, 2.5) {};
		\node [style=none] (66) at (7.5, -1.5) {};
		\node [style=G-bdry, label={below:$\partial D$}] (77) at (4, -2) {};
		\node [style=none] (78) at (4.5, -2) {};
		\node [style=none] (79) at (5, -2) {};
		\node [style=none] (80) at (5.5, -2) {};
		\node [style=none] (81) at (4.75, -0.5) {};
		\node [style=none] (82) at (5.25, -0.5) {};
		\node [style=none] (83) at (5.75, -0.5) {};
		\node [style=none] (84) at (4.75, -3.5) {};
		\node [style=none] (85) at (5.25, -3.5) {};
		\node [style=none] (86) at (5.75, -3.5) {};
		\node [style=none] (88) at (0.5, 2) {};
		\node [style=none] (89) at (1.5, -2) {};
		\node [style=empty, label={below:$\phi_*(\alpha)$}] (19) at (6, -2) {};
		\node [style=none] (90) at (-3, 2) {};
		\node [style=none] (91) at (-5.5, 0.75) {};
		\node [style=none] (92) at (-4, 1.875) {};
		\node [style=none] (93) at (-4.75, 1.5) {};
		\node [style=none] (94) at (-3, -2) {};
		\node [style=none] (95) at (-5.5, -0.75) {};
		\node [style=none] (96) at (-4, -1.875) {};
		\node [style=none] (97) at (-4.75, -1.5) {};
		\node [style=G{n}, label={left:$\Gamma_0$}] (98) at (-6.25, 0) {};
	\end{pgfonlayer}
	\begin{pgfonlayer}{edgelayer}
		\draw [style=mydashed] (15)
			 to [in=-45, out=180] (33.center)
			 to [in=-45, out=135, looseness=1.75] (26.center)
			 to [in=-45, out=135, looseness=1.75] (94.center)
			 to [in=-45, out=135, looseness=1.25] (96.center)
			 to [in=0, out=135, looseness=1.50] (97.center)
			 to [in=0, out=180, looseness=1.25] (95.center)
			 to [in=-45, out=180, looseness=1.50] (98);
		\draw [style=grey edge] (43.center) to (15);
		\draw [style=grey edge] (15) to (45.center);
		\draw [style=grey edge] (42.center) to (15);
		\draw [style=grey edge] (15) to (44.center);
		\draw [style=G-domain] (19) to (60.center);
		\draw [style=G-domain] (19) to (59.center);
		\draw [style=G-domain] (19) to (61.center);
		\draw [style=grey edge] (48.center) to (2);
		\draw [style=grey edge] (2) to (50.center);
		\draw [style=grey edge] (47.center) to (2);
		\draw [style=grey edge] (2) to (49.center);
		\draw [style=grey edge] (2) to (64.center);
		\draw [style=grey edge] (65.center) to (2);
		\draw [style=grey edge] (38.center) to (1);
		\draw [style=grey edge] (1) to (40.center);
		\draw [style=grey edge] (37.center) to (1);
		\draw [style=grey edge] (1) to (39.center);
		\draw [style=G-domain] (66.center) to (19);
		\draw (77) to (19);
		\draw [style=G-domain] (81.center) to (78.center);
		\draw [style=G-domain] (82.center) to (79.center);
		\draw [style=G-domain] (80.center) to (83.center);
		\draw [style=G-domain] (80.center) to (86.center);
		\draw [style=G-domain] (79.center) to (85.center);
		\draw [style=G-domain] (78.center) to (84.center);
		\draw [style=thick arrow] (2) to node [midway, above] {$\bvec u$} (88.center);
		\draw [style=thick arrow] (15) to node [midway, above] {$\phi_\#(\bvec u)$} (89.center);
		\draw [style=mydashed] (1)
			 to [in=45, out=-180] (90.center)
			 to [in=45, out=-135, looseness=1.25] (92.center)
			 to [in=0, out=-135, looseness=1.50] (93.center)
			 to [in=0, out=-180, looseness=1.25] (91.center)
			 to [in=45, out=-180, looseness=1.50] (98);
		\draw [style=mydashed] (15)
			 to [in=135, out=0, looseness=1.25] (111.center)
			 to [in=135, out=-45, looseness=1.50] (114.center)
			 to [in=135, out=-45, looseness=1.50] (113.center)
			 to [in=180, out=-45, looseness=1.25] (77);
		\draw [style=grey edge] (106.center) to (77);
		\draw [style=grey edge] (77) to (108.center);
		\draw [style=grey edge] (105.center) to (77);
		\draw [style=grey edge] (77) to (107.center);
		\draw [style=mydashed] (1)
			 to [in=-135, out=0, looseness=1.25] (117.center)
			 to [in=-135, out=45, looseness=1.50] (120.center)
			 to [in=-135, out=45, looseness=1.50] (119.center)
			 to [in=-180, out=45, looseness=1.25] (2);
	\end{pgfonlayer}
\end{tikzpicture}
  \caption{Contrary to the claim, with
 $\alpha \in {\gamcol\tilde\Gamma_0}$, ${\gamdomcol D} \in \mathcal D_1$ and $\partial D \in {\smgamcol\tilde\Gamma_1}$.}
  \label{fig:claimzerodirn:contra}
\end{subfigure}%
\caption{\autoref{clm:zerodirn}}
\label{fig:claimzerodirn}
\end{figure}

\begin{proof}[Proof of \autoref{clm:zerodirn}]
Let $\alpha$ be the (unique) nearest point on $\Hull(\tilde \Gamma_0)$ to $\zeta$. We first show that $\alpha \in \tilde \Gamma_0$.
If $\alpha$ is a vertex of valency at least $3$ in $\Hull(\tilde \Gamma_0 \cup \set \zeta) \subset \bigcup_{m \le m_0}\T_m$, then it must belong to $\GV_{m_0}$. Hence, or otherwise if $\alpha$ is an endpoint of $\Hull(\tilde \Gamma_0)$, we have $\alpha \in \tilde \Gamma_0$ because $\tilde \Gamma_0 = \Hull(\tilde \Gamma_0) \cap \GV_{m_0}$.

 Observe that $\bvec u = \vec v(\alpha) = \vec v(\tilde \Gamma_0) = \vec v(\Gamma_0)$. 
Let $\bvec v$ be any other direction at $\zeta$ and $\beta \in \bvec v$. Clearly $(\alpha, \beta)$ is disjoint from $\Hull(\tilde \Gamma_0) \supset T$ and the direction $\vec v(\zeta)$ at $\alpha$ does not intersect $T$. Then by \autoref{lem:nofolding}, $\phi_*$ maps $(\alpha, \beta)$ injectively; hence $\phi_\#(\bvec u) = \vec v(\phi_*(\alpha))$ and $\phi_\#(\bvec v) = \vec v(\phi_*(\beta))$ are distinct directions at $\phi_*(\zeta)$. This also means $\bvec v = \vec v(\beta)$ is a good direction. Below we will show that $\phi_\#(\bvec u) = \vec v(\Gamma_0)$, and therefore $\phi_\#(\bvec v) \ne \vec v(\Gamma_0)$.
 
 Suppose that $\phi_\#(\bvec u) \ne \vec v(\Gamma_0)$ at $\phi_*(\zeta)$; see \autoref{fig:claimzerodirn:contra}. By \autoref{clm:main} and because $\alpha \in \Gamma_1$, either $\phi_*(\alpha) \in \Gamma_1$ or $\phi_*(\alpha) \in D$ where $D \in \mathcal D_1$ is a persistent F-disk. In the former case, we have $\phi_*(\zeta)$ lying between $\Gamma_0$ and $\phi_*(\alpha) \in \Gamma_1$, so $\phi_*(\zeta) \in \tilde \Gamma_1$, against our hypotheses. In the latter case, we have that $[\phi_*(\zeta), \phi_*(\alpha)) = \phi_*([\zeta, \alpha))$ is a path from outside $D$ to inside $D$, so it contains $\partial D$. By the persistent F-disk axiom \ref{ax:bdry}, $\partial D \in \tilde\Gamma_1$, therefore $\phi_*(\zeta) \in \tilde \Gamma_1$ by $m_0$-convexity, contradicting our hypotheses.  
\end{proof}

\textbf{Case A)} Choose $n$ large enough such that for each classical critical point $\gamma \in \Crit(\phi_*)$, every $\delta \in \Orb_{\phi_*}^+(\gamma)$ in its orbit lies in a persistent F-disk if they ever will during the procedure. This is possible in finitely many steps because if for some $n'$ a single $\phi_*^j(\gamma) \in \Orb_{\phi_*}^+(\gamma)$ lies in a persistent F-disk of $\mathcal D_{n'}$, then (by \ref{ax:iter}) we have every member of $\Orb_{\phi_*}^+(\phi_*^j(\gamma))$ in persistent F-disks, leaving only finitely many other iterates $\gamma, \phi_*(\gamma), \dots, \phi_*^{j-1}(\gamma)$ to consider.

Let $\xi \in \Gamma_{n+1}$ be an \unflanked{} satellite point of multiplicity $m_+$, whose inclusion to the set was triggered by applying the rules to some $\zeta \in \tilde \Gamma_n$. More precisely, for some $j >0$, $\xi = \phi_*^j(\zeta)$ and $\phi_*^t(\zeta) \in \Gamma_{n+1}$ for each $0 \le t \le j$. We may assume, by replacing $\zeta, \xi$ if necessary, that $\phi_*^j(\zeta) \in \Gamma_{n+1}\sm \tilde\Gamma_n$ for each $0 < t \le j$, and $\xi$ is the first \unflanked{} point of multiplicity $m_+$ in this orbit.

We argue that $\zeta \nin \Gamma_n$. Suppose not, then we already applied the rules to it in step $n$. Given that $\phi_*(\zeta) \nin \Gamma_n$, we can conclude from \autoref{clm:main} that $\phi_*(\zeta)$ must lie in an F-disk $U \in \mathcal{D}_n \subseteq \mathcal{D}_{n+1}$, implying that $\phi_*(\zeta), \dots, \phi_*^j(\zeta)$ also lie in F-disks, not in $\Gamma_{n+1}$.

Note that $\mm(\zeta) \le m_+$ by our assumption on multiplicities for large $n$. Further, we claim that $\mm(\phi_*^t(\zeta)) = m_+$ for each $0 \le t \le j$. Indeed, \autoref{prop:skewstab:multiplicitydivides} says that \[\mm(\zeta) \ge \mm(\phi_*(\zeta)) \ge \mm(\phi_*^2(\zeta)) \ge \cdots \ge \mm(\phi_*^j(\zeta)) = \mm(\xi) = m_+.\]
Given that $\xi$ is \unflanked{} by $\Gamma_{n+1}$, it must be satellite, meaning $\mg(\xi) > \mm(\xi) = m_+$; see \autoref{rmk:skewstab:specialgeneric}. We also find that $\phi_*^t(\zeta)$ is satellite for every $0 \le t \le j$ because the proposition states a similar inequality for generic multiplicities 
 \[\mg(\zeta) \ge \mg(\phi_*(\zeta)) \ge \mg(\phi_*^2(\zeta)) \ge \cdots \ge \mg(\phi_*^j(\zeta)) = \mg(\xi) > m_+.\]
 
 Now let $\alpha, \beta$ be the nearest vertices of $\T_{m_+}$, i.e.\ they are free with $\mg(\alpha) = \mg(\beta) = m_+$. We may assume, by swapping roles of $\alpha, \beta$ if necessary, that $\tilde\Gamma_0$ is disjoint from the direction $\vec v(\beta)$ at $\alpha$; indeed $\Gamma_0 \cap \T_{m_+}$ contains a point $\zeta_0$ with generic multiplicity $1$. 
 If $(\alpha, \beta)$ intersected $T \subseteq \Hull(\Gamma_0)$ then we would have $\alpha, \beta \in \tilde \Gamma_0$ by construction of the smooth convex hull, so $\zeta \in \tilde \Gamma_0$, contradicting the assumption that $\zeta \in \tilde \Gamma_n \sm \Gamma_n$ ($n \ge 1$). Thus by \autoref{lem:nofolding}, $(\alpha, \beta)$ maps injectively to $(\phi_*(\alpha), \phi_*(\beta))$.
 
\begin{clm}\label{clm:mapintervals}
 For each $0 \le t \le j$, $(\phi_*^t(\zeta), \phi_*^t(\beta))$ is an interval which $\phi_*$ maps homeomorphically to \sloppy\mbox{$(\phi_*^{t+1}(\zeta), \phi_*^{t+1}(\beta))$}. Furthermore, $\Gamma_0$ lies in a direction other than $\vec v(\phi_*^t(\beta))$ at $\phi_*^t(\zeta)$. See \autoref{fig:nofolds3}.
\end{clm}
  
 \begin{figure}
  \centering
   \input{nofolds3.tikz}
  \caption{Vertices of {\gamcol$\Gamma_n$}, {\smgamcol$\tilde\Gamma_n$}, {\gampluscol$\Gamma_{n+1}$}, and persistent F-disks of {\gamdomcol$\mathcal D_{n+1}$}.}
  \label{fig:nofolds3}
\end{figure}%

\begin{proof}
 We use induction to prove the claim. The base case was completed above. Assume that claim holds for some $0 \le t < j$. Let $\bvec u = \vec v(\Gamma_0)$ at $\phi_*^t(\zeta)$, which is distinct from $\vec v(\phi_*^t(\beta))$. Provided that $n > 1$, and $\phi_*^{t+1}(\zeta)$ does not lie in an F-disk, $D \in \mathcal D_{n+1}$, $\phi_*^t(\zeta)$ satisfies the hypotheses of \autoref{clm:zerodirn}. Thus $\phi_\#(\bvec u) = \vec v(\Gamma_0) \ne \phi_\#(\vec v(\phi_*^t(\beta))) = \vec v(\phi_*^{t+1}(\beta))$ at $\phi_*^{t+1}(\zeta)$. Thus by \autoref{lem:nofolding}, $\phi_*$ maps $(\phi_*^t(\zeta), \phi_*^t(\beta))$ injectively to $(\phi_*^{t+1}\zeta), \phi_*^{t+1}(\beta))$.
\end{proof}

Recall that $\xi$ is \unflanked{} by $\Gamma_{n+1}$. Now, using \autoref{clm:mapintervals}, we must have $\phi_*^j(\beta) \nin \Gamma_{n+1}$ because $\Gamma_0 \subset \Gamma_{n+1}$ lies in a different direction. 
Applying \autoref{clm:main} inductively to $\beta \in \tilde \Gamma_n$, there is a $j' \le j$ such that $\phi_*^t(\beta) \in \Gamma_{n+1}$ for every $t < j'$ and $\phi_*^{j'}(\beta) \in D'$ for some $D' \in \mathcal D_{n+1}$. Further, by the axiom \ref{ax:iter} of persistent F-disks, $\phi_*^{j-j'}(D') \subset D$ for some persistent F-disk $D \in \mathcal D_{n+1}$. Therefore $\phi_*^j(\beta) \in D$.

We know that $\xi \nin D$ since $D \cap \Gamma_{n+1} = \emp$. Therefore $[\phi_*^j(\zeta), \phi_*^j(\beta))$ is a path from outside $D$ to inside $D$, so it contains $\partial D$. 

Using \ref{ax:gen} and \autoref{prop:skewstab:multiplicitydivides} we obtain a chain of inequalities \[\mm(\partial D) \le \mg(\partial D) = \mm(D) \le \mm(\phi_*^j(\beta)) \le \mg(\phi_*^j(\beta)) \le \mg(\beta) = m_+.\]
On the other hand, since $\mm(\xi) = m_+$ and $\xi$ lies between $\Gamma_0$ and $\partial D$, we have $\mm(\partial D) \ge m_+$; thus all the relations above are equalities. Therefore $\phi_*^j(\beta)$ and $\partial D$ are free points of multiplicity $m_+$. It follows that $d_\bH(\partial D, \phi_*^j(\beta))$ is a positive multiple of $\frac 1{m_+}$.

 \begin{figure}
  \centering
   \resizebox{\linewidth}{!}{\input{nofolds4.tikz}}
  \caption{Vertices of {\gamcol$\Gamma_n$}, {\smgamcol$\tilde\Gamma_n$}, {\gampluscol$\Gamma_{n+1}$},  persistent F-disks of {\gamdomcol$\mathcal D_{n+1}$}, and $\Ram(\phi)$ (bold).}
  \label{fig:nofolds4}
\end{figure}%

Consider the mapping $\phi_*^j : [\zeta, \beta) \to [\xi, \phi_*^j(\beta))$, which is a homeomorphism by \autoref{clm:mapintervals}.
Let $\zeta' \in [\zeta, \beta)$ be the unique preimage of $\partial D$ by $\phi_*^j$. 
The length of the interval $(\zeta, \beta)$ is relatively small, $d_\bH(\zeta, \beta) < d_\bH(\alpha, \beta) = \frac 1{m_+}$, so the interval $[\zeta', \beta]$ expanded in length under $\phi_*^j$. 
Hence, for some $0 \le l < j$, $(\phi_*^l(\zeta'), \phi_*^l(\beta))$ contains a non-trivial subinterval of $\Ram(\phi) \sm T$, using the contrapositive of \cite[Corollary 3.58]{thesis}.  Furthermore, the component of $\Ram(\phi) \sm T$ in question takes the form $[\alpha', c]$ where $\alpha'$ is an endpoint of $T$ and $c \in \Crit(\phi)$. 
Therefore the direction $\bvec w = \vec v(\phi_*^l(\beta))$ at $\phi_*^l(\zeta')$ contains $c$. By \autoref{clm:zerodirn} the iterates of $\bvec w$ are good directions, so $\phi_*^{j-l}(\bvec w) = \vec v(\phi_*^j(\beta))$, which is precisely $D$. Thus the persistent F-disk $D$ contains $\phi_*^{j-l}(c)$. See \autoref{fig:nofolds4}. Given our choice of sufficiently large $n$ made at the beginning of (Case A), this is an `older' persistent F-disk, meaning $D \in \mathcal{D}_n$. In the latter case, $\partial D$ already belongs to $\tilde \Gamma_n$; thus we conclude that $\xi$ is \flanked{} in $\Gamma_{n+1}$, namely by $\partial D$ in one direction and by $\Gamma_0$ in the other.
 
 \textbf{Case B)} Now suppose that some $\zeta \in \tilde \Gamma_n$ triggered the addition of a multiplicity $m_+$ point $\phi_*(\zeta)$ to $\Gamma_{n+1} \sm \tilde \Gamma_n$ through one of the rules. By the same reasoning as in Case (A), we may assume $\zeta \in \tilde \Gamma_n \sm \Gamma_n$. Given $\mm(\zeta) \le m_+$ and $\mm(\zeta) \ge \mm(\phi_*(\zeta))$ by \autoref{prop:skewstab:multiplicitydivides}, we have $\mm(\zeta) = m_+$ also.
 
 Suppose that $\zeta \nin \Hull(\Gamma_n)$, then according to \autoref{prop:skewstab:smoothhull} $\zeta$ is part of a leaf added as a result of some $\xi \in \Hull(\Gamma_n) \cap \GV_{m_0}$ not being \flanked{}. Specifically, $\xi$ is satellite, with $\vec v(\zeta)$ as a special direction which is disjoint from $\Gamma_n$. The points on $[\xi, \zeta]$ have the same multiplicity, $m_+$. We assume every point of $\Gamma_n$ is \flanked{}, so $\xi \nin \Gamma_n$. Therefore we can find points of $\Gamma_n$ in two distinct directions at $\xi$, however neither direction can be $\vec v(\zeta)$. Since there are only two special directions at $\xi$, we can deduce there is a generic direction $\bvec v$ at $\xi$ and $\xi' \in \bvec v \cap \Gamma_n$; see \autoref{fig:weirdunflanked}. Hence $\mm(\xi') \ge \mm(\bvec v) = \mg(\xi) > \mm(\xi) = m_+$.  Because $\xi \in \tilde \Gamma_n \sm \Gamma_n$, we may assume that $\xi'$ was new in $\Gamma_n \sm \tilde \Gamma_{n-1}$. This contradicts our assumption on the maximum multiplicity of new points in $\Gamma_n$.
 
   \begin{figure}
  \centering
   \begin{tikzpicture}
	\begin{pgfonlayer}{nodelayer}
		\node [style=none] (90) at (-6.25, -0.75) {};
		\node [style=none] (91) at (-5.75, -0.75) {};
		\node [style=none] (92) at (-5.75, 0.75) {};
		\node [style=none] (93) at (-6.25, 0.75) {};
		\node [style=none] (17) at (-3.25, -0.75) {};
		\node [style=none] (18) at (-2.75, -0.75) {};
		\node [style=none] (19) at (-2.75, 0.75) {};
		\node [style=none] (20) at (-3.25, 0.75) {};
		\node [style=none] (5) at (0.75, -0.25) {};
		\node [style=none] (6) at (0.75, 0.25) {};
		\node [style={G{n}}] (0) at (0, 3) {};
		\node [style=smooth n, label={right:$\zeta$}] (1) at (0, -3) {};
		\node [style={1/2 smooth n}, label={right:$\xi$}] (2) at (0, 0) {};
		\node [style={1/2G{n}q}, label={left:$\xi'$}] (4) at (-6, 0) {};
		\node [style=none] (7) at (1, 3) {};
		\node [style=none] (8) at (1, -3) {};
		\node [style=none] (9) at (-1, -3) {};
		\node [style=none] (10) at (-1, 3) {};
		\node [style={1/2 smooth n}] (11) at (-3, 0) {};
		\node [style=none] (21) at (0, -4) {};
		\node [style=none] (22) at (0, 4) {};
		\node [style={1/3 smooth n}] (24) at (0, -1) {};
		\node [style=none] (27) at (0.5, -1) {};
		\node [style=none] (28) at (-0.5, -1) {};
		\node [style=none] (29) at (0.25, -1.25) {};
		\node [style=none] (30) at (0.25, -0.75) {};
		\node [style=none] (31) at (-0.25, -0.75) {};
		\node [style=none] (32) at (-0.25, -1.25) {};
		\node [style={1/3 smooth n}] (33) at (0, 1) {};
		\node [style=none] (34) at (0.5, 1) {};
		\node [style=none] (35) at (-0.5, 1) {};
		\node [style=none] (36) at (0.25, 0.75) {};
		\node [style=none] (37) at (0.25, 1.25) {};
		\node [style=none] (38) at (-0.25, 1.25) {};
		\node [style=none] (39) at (-0.25, 0.75) {};
		\node [style={1/3 smooth n}] (40) at (-4, 0) {};
		\node [style=none] (41) at (-4, 0.5) {};
		\node [style=none] (42) at (-4, -0.5) {};
		\node [style=none] (43) at (-4.25, 0.25) {};
		\node [style=none] (44) at (-3.75, 0.25) {};
		\node [style=none] (45) at (-3.75, -0.25) {};
		\node [style=none] (46) at (-4.25, -0.25) {};
		\node [style={1/3 smooth n}] (47) at (-2, 0) {};
		\node [style=none] (48) at (-2, 0.5) {};
		\node [style=none] (49) at (-2, -0.5) {};
		\node [style=none] (50) at (-2.25, 0.25) {};
		\node [style=none] (51) at (-1.75, 0.25) {};
		\node [style=none] (52) at (-1.75, -0.25) {};
		\node [style=none] (53) at (-2.25, -0.25) {};
		\node [style={1/4 smooth n}] (54) at (-1.5, 0) {};
		\node [style=none] (55) at (-1.25, 0.125) {};
		\node [style=none] (56) at (-1.75, 0.125) {};
		\node [style=none] (57) at (-1.375, -0.25) {};
		\node [style=none] (58) at (-1.375, 0.25) {};
		\node [style=none] (59) at (-1.625, 0.25) {};
		\node [style=none] (60) at (-1.625, -0.25) {};
		\node [style=none] (61) at (-1.25, -0.15) {};
		\node [style=none] (62) at (-1.75, -0.15) {};
		\node [style={1/4 smooth n}] (63) at (0, -1.5) {};
		\node [style=none] (64) at (0.25, -1.375) {};
		\node [style=none] (65) at (-0.25, -1.375) {};
		\node [style=none] (66) at (0.125, -1.75) {};
		\node [style=none] (67) at (0.125, -1.25) {};
		\node [style=none] (68) at (-0.125, -1.25) {};
		\node [style=none] (69) at (-0.125, -1.75) {};
		\node [style=none] (70) at (0.25, -1.65) {};
		\node [style=none] (71) at (-0.25, -1.65) {};
		\node [style={1/4 smooth n}] (72) at (0, 1.5) {};
		\node [style=none] (73) at (0.25, 1.625) {};
		\node [style=none] (74) at (-0.25, 1.625) {};
		\node [style=none] (75) at (0.125, 1.25) {};
		\node [style=none] (76) at (0.125, 1.75) {};
		\node [style=none] (77) at (-0.125, 1.75) {};
		\node [style=none] (78) at (-0.125, 1.25) {};
		\node [style=none] (79) at (0.25, 1.35) {};
		\node [style=none] (80) at (-0.25, 1.35) {};
		\node [style={1/4 smooth n}] (81) at (-4.5, 0) {};
		\node [style=none] (82) at (-4.25, 0.125) {};
		\node [style=none] (83) at (-4.75, 0.125) {};
		\node [style=none] (84) at (-4.375, -0.25) {};
		\node [style=none] (85) at (-4.375, 0.25) {};
		\node [style=none] (86) at (-4.625, 0.25) {};
		\node [style=none] (87) at (-4.625, -0.25) {};
		\node [style=none] (88) at (-4.25, -0.15) {};
		\node [style=none] (89) at (-4.75, -0.15) {};
	\end{pgfonlayer}
	\begin{pgfonlayer}{edgelayer}
		\draw (0) to (2);
		\draw (2) to (1);
		\draw [style=grey edge] (6.center) to (2);
		\draw [style=grey edge] (5.center) to (2);
		\draw [style=grey edge] (10.center) to (0);
		\draw [style=grey edge] (0) to (7.center);
		\draw [style=grey edge] (9.center) to (1);
		\draw [style=grey edge] (1) to (8.center);
		\draw (2) to (11);
		\draw (11) to (4);
		\draw [style=grey edge] (18.center) to (11);
		\draw [style=grey edge] (11) to (20.center);
		\draw [style=grey edge] (17.center) to (11);
		\draw [style=grey edge] (11) to (19.center);
		\draw [style=grey edge] (22.center) to (0);
		\draw [style=grey edge] (21.center) to (1);
		\draw [style=grey edge] (30.center) to (24);
		\draw [style=grey edge] (24) to (32.center);
		\draw [style=grey edge] (29.center) to (24);
		\draw [style=grey edge] (24) to (31.center);
		\draw [style=grey edge] (27.center) to (24);
		\draw [style=grey edge] (24) to (28.center);
		\draw [style=grey edge] (37.center) to (33);
		\draw [style=grey edge] (33) to (39.center);
		\draw [style=grey edge] (36.center) to (33);
		\draw [style=grey edge] (33) to (38.center);
		\draw [style=grey edge] (34.center) to (33);
		\draw [style=grey edge] (33) to (35.center);
		\draw [style=grey edge] (44.center) to (40);
		\draw [style=grey edge] (40) to (46.center);
		\draw [style=grey edge] (43.center) to (40);
		\draw [style=grey edge] (40) to (45.center);
		\draw [style=grey edge] (41.center) to (40);
		\draw [style=grey edge] (40) to (42.center);
		\draw [style=grey edge] (51.center) to (47);
		\draw [style=grey edge] (47) to (53.center);
		\draw [style=grey edge] (50.center) to (47);
		\draw [style=grey edge] (47) to (52.center);
		\draw [style=grey edge] (48.center) to (47);
		\draw [style=grey edge] (47) to (49.center);
		\draw [style=grey edge] (58.center) to (54);
		\draw [style=grey edge] (54) to (60.center);
		\draw [style=grey edge] (57.center) to (54);
		\draw [style=grey edge] (54) to (59.center);
		\draw [style=grey edge] (55.center) to (54);
		\draw [style=grey edge] (54) to (56.center);
		\draw [style=grey edge] (61.center) to (54);
		\draw [style=grey edge] (62.center) to (54);
		\draw [style=grey edge] (67.center) to (63);
		\draw [style=grey edge] (63) to (69.center);
		\draw [style=grey edge] (66.center) to (63);
		\draw [style=grey edge] (63) to (68.center);
		\draw [style=grey edge] (64.center) to (63);
		\draw [style=grey edge] (63) to (65.center);
		\draw [style=grey edge] (70.center) to (63);
		\draw [style=grey edge] (71.center) to (63);
		\draw [style=grey edge] (76.center) to (72);
		\draw [style=grey edge] (72) to (78.center);
		\draw [style=grey edge] (75.center) to (72);
		\draw [style=grey edge] (72) to (77.center);
		\draw [style=grey edge] (73.center) to (72);
		\draw [style=grey edge] (72) to (74.center);
		\draw [style=grey edge] (79.center) to (72);
		\draw [style=grey edge] (80.center) to (72);
		\draw [style=grey edge] (85.center) to (81);
		\draw [style=grey edge] (81) to (87.center);
		\draw [style=grey edge] (84.center) to (81);
		\draw [style=grey edge] (81) to (86.center);
		\draw [style=grey edge] (82.center) to (81);
		\draw [style=grey edge] (81) to (83.center);
		\draw [style=grey edge] (88.center) to (81);
		\draw [style=grey edge] (89.center) to (81);
		\draw [style=grey edge] (91.center) to (4);
		\draw [style=grey edge] (4) to (93.center);
		\draw [style=grey edge] (90.center) to (4);
		\draw [style=grey edge] (4) to (92.center);
	\end{pgfonlayer}
\end{tikzpicture}
  \caption{$\xi \in {\smgamcol\tilde\Gamma_n} \sm {\gamcol\Gamma_n}$ is \unflanked{} in $\GV_{m_0} \cap \Hull({\gamcol \Gamma_n})$ making $\zeta \in {\smgamcol\tilde\Gamma_n}$ required. An example imagined with $m_0  = m_+ = 4$.} 
  \label{fig:weirdunflanked}
\end{figure}
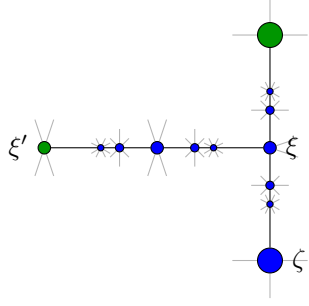%
   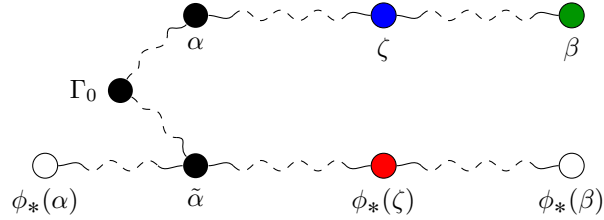
\begin{figure}
  \centering
   \begin{tikzpicture}
	\begin{pgfonlayer}{nodelayer}
		\node [style=empty, label={below:$\phi_*(\alpha)$}] (83) at (-11, -5) {};
		\node [style=none] (85) at (-10, -5) {};
		\node [style=none] (86) at (-9, -5) {};
		\node [style=none] (87) at (-8, -5) {};
		\node [style=none] (60) at (2, -5) {};
		\node [style=none] (56) at (2, -1) {};
		\node [style=none] (41) at (-7.5, -4.25) {};
		\node [style=none] (43) at (-8.25, -3.5) {};
		\node [style={G{n}}, label={below:$\beta$}] (0) at (3, -1) {};
		\node [style=smooth n, label={below:$\zeta$}] (2) at (-2, -1) {};
		\node [style=empty, label={below:$\phi_*(\beta)$}] (5) at (3, -5) {};
		\node [style={G{n+1}}, label={below:$\phi_*(\zeta)$}] (7) at (-2, -5) {};
		\node [style=none] (15) at (-6, -1) {};
		\node [style=black vertex, label={below:$\tilde \alpha$}] (28) at (-7, -5) {};
		\node [style=black vertex, label={left:$\Gamma_0$}] (26) at (-9, -3) {};
		\node [style=black vertex, label={below:$\alpha$}] (27) at (-7, -1) {};
		\node [style=none] (31) at (-7.75, -1.5) {};
		\node [style=none] (35) at (-8.5, -2.25) {};
		\node [style=none] (47) at (-6, -5) {};
		\node [style=none] (63) at (-1, -5) {};
		\node [style=none] (67) at (-5, -5) {};
		\node [style=none] (69) at (-4, -5) {};
		\node [style=none] (70) at (-3, -5) {};
		\node [style=none] (71) at (-5, -1) {};
		\node [style=none] (73) at (-4, -1) {};
		\node [style=none] (74) at (-3, -1) {};
		\node [style=none] (75) at (-1, -1) {};
		\node [style=none] (77) at (0, -1) {};
		\node [style=none] (78) at (1, -1) {};
		\node [style=none] (80) at (0, -5) {};
		\node [style=none] (81) at (1, -5) {};
	\end{pgfonlayer}
	\begin{pgfonlayer}{edgelayer}
		\draw [style=mydashed] (26)
			 to [in=-90, out=60, looseness=1.50] (35.center)
			 to [in=-90, out=90, looseness=1.50] (31.center);
		\draw [style=mydashed] (26)
			 to [in=180, out=-30, looseness=1.50] (43.center)
			 to [in=180, out=0, looseness=1.50] (41.center);
		\draw [in=135, out=0, looseness=1.50] (41.center) to (28);
		\draw [in=-135, out=90, looseness=1.50] (31.center) to (27);
		\draw [in=135, out=0] (28) to (47.center);
		\draw [in=-135, out=0] (27) to (15.center);
		\draw [in=45, out=-180] (0) to (56.center);
		\draw [in=-45, out=180] (5) to (60.center);
		\draw [style=mydashed] (47.center)
			 to [in=135, out=-45, looseness=1.75] (67.center)
			 to [in=135, out=-45, looseness=1.75] (69.center)
			 to [in=135, out=-45, looseness=1.75] (70.center);
		\draw [style=mydashed] (15.center)
			 to [in=-135, out=45, looseness=1.75] (71.center)
			 to [in=-135, out=45, looseness=1.75] (73.center)
			 to [in=-135, out=45, looseness=1.75] (74.center);
		\draw [in=45, out=-180] (2) to (74.center);
		\draw [style=mydashed] (75.center)
			 to [in=-135, out=45, looseness=1.75] (77.center)
			 to [in=-135, out=45, looseness=1.75] (78.center)
			 to [in=-135, out=45, looseness=1.75] (56.center);
		\draw [style=mydashed] (63.center)
			 to [in=135, out=-45, looseness=1.75] (80.center)
			 to [in=135, out=-45, looseness=1.75] (81.center)
			 to [in=135, out=-45, looseness=1.75] (60.center);
		\draw [in=-45, out=180] (7) to (70.center);
		\draw [in=-135, out=0] (2) to (75.center);
		\draw [in=135, out=0] (7) to (63.center);
		\draw [style=mydashed] (85.center)
			 to [in=135, out=-45, looseness=1.75] (86.center)
			 to [in=135, out=-45, looseness=1.75] (87.center);
		\draw [in=-45, out=180] (28) to (87.center);
		\draw [in=135, out=0] (83) to (85.center);
	\end{pgfonlayer}
\end{tikzpicture}
  \caption{Vertices of $\tilde \Gamma_0$, {\gamcol$\Gamma_n$}, {\smgamcol$\tilde\Gamma_n$}, {\gampluscol$\Gamma_{n+1}$}.} 
  \label{fig:nofold1}
\end{figure}%
 Therefore, $\zeta$ has points of $\Gamma_n$ in two directions. First choose $\alpha$ to be the closest point in $\tilde\Gamma_0$. 
Assuming $n > 1$, we can assume that $\zeta \nin \Hull(\Gamma_0)$. In the second direction we can find $\beta \in \Gamma_n$, so $\zeta \in (\alpha, \beta)$. 
By \autoref{lem:nofolding}, this interval is mapped by $\phi_*$ injectively, so $\phi_*(\zeta) \in \phi_*((\alpha, \beta)) = (\phi_*(\alpha), \phi_*(\beta))$. Furthermore, by \autoref{clm:zerodirn}, $\Gamma_0$ lies in a direction $\vec v(\phi_*(\alpha))\ne \vec v(\phi_*(\beta))$ at $\phi_*(\zeta)$; now pick $\tilde \alpha \in \tilde\Gamma_0$ arbitrarily\footnote{It is not necessary but one can check that if $\tilde \alpha$ is chosen nearest to $\phi_*(\zeta)$ then $\tilde \alpha \in [\phi_*(\alpha), \phi_*(\zeta))$).}. See \autoref{fig:nofold1}. 
If $\phi_*(\beta) \in \Gamma_n$, then write $\tilde \beta = \phi_*(\beta)$; see \autoref{fig:caseB:1}. Otherwise suppose that $\phi_*(\beta) \nin \Gamma_n$, then by the claim, $\phi_*(\beta)$ lies in an F-disk $D \in \mathcal{D}_n$; write $\tilde \beta = \partial D$. We know that $\phi_*(\zeta) \nin D$ since otherwise it would not have been added to $\Gamma_{n+1}$. Thus $[\phi_*(\zeta), \phi_*(\beta)]$ is a path from outside to inside $D$, so it contains $\tilde \beta$; see \autoref{fig:caseB:2}. By the axioms of $\mathcal{D}_n$, we have $\partial D \in \tilde\Gamma_n$, so $\tilde\beta \in [\phi_*(\zeta), \phi_*(\beta)] \cap \tilde \Gamma_n$.  We have shown that $\phi_*(\zeta) \in [\tilde\alpha, \tilde\beta]$ and thus $\phi_*(\zeta) \in \tilde\Gamma_n$, contradicting our assumption.

\newlength\FullTextWidth
\setlength\FullTextWidth\textwidth

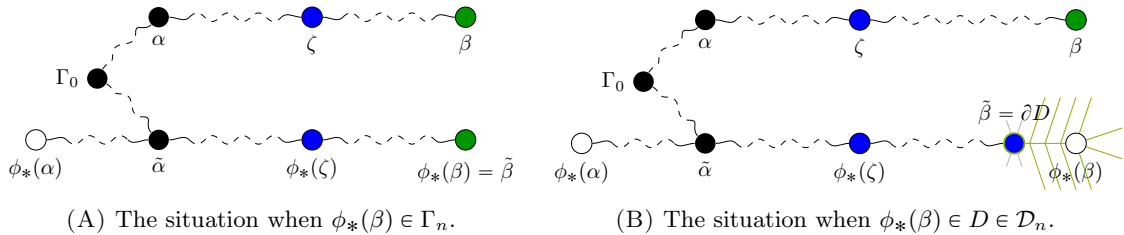
\begin{figure}
\centering
\begin{subfigure}{.47\textwidth}
  \centering
   \resizebox{\textwidth}{!}{\begin{tikzpicture}
	\begin{pgfonlayer}{nodelayer}
		\node [style=none] (60) at (2, -5) {};
		\node [style=none] (56) at (2, -1) {};
		\node [style={G{n}}, label={below:$\beta$}] (0) at (3, -1) {};
		\node [style=smooth n, label={below:$\zeta$}] (2) at (-2, -1) {};
		\node [style={G{n}}, label={below:$\phi_*(\beta) = \tilde\beta$}] (5) at (3, -5) {};
		\node [style=smooth n, label={below:$\phi_*(\zeta)$}] (7) at (-2, -5) {};
		\node [style=none] (15) at (-6, -1) {};
		\node [style=black vertex, label={below:$\alpha$}] (27) at (-7, -1) {};
		\node [style=none] (31) at (-7.75, -1.5) {};
		\node [style=none] (35) at (-8.5, -2.25) {};
		\node [style=none] (47) at (-6, -5) {};
		\node [style=none] (63) at (-1, -5) {};
		\node [style=none] (67) at (-5, -5) {};
		\node [style=none] (69) at (-4, -5) {};
		\node [style=none] (70) at (-3, -5) {};
		\node [style=none] (71) at (-5, -1) {};
		\node [style=none] (73) at (-4, -1) {};
		\node [style=none] (74) at (-3, -1) {};
		\node [style=none] (75) at (-1, -1) {};
		\node [style=none] (77) at (0, -1) {};
		\node [style=none] (78) at (1, -1) {};
		\node [style=none] (80) at (0, -5) {};
		\node [style=none] (81) at (1, -5) {};
		\node [style=empty, label={below:$\phi_*(\alpha)$}] (82) at (-11, -5) {};
		\node [style=none] (83) at (-10, -5) {};
		\node [style=none] (84) at (-9, -5) {};
		\node [style=none] (85) at (-8, -5) {};
		\node [style=none] (86) at (-7.5, -4.25) {};
		\node [style=none] (87) at (-8.25, -3.5) {};
		\node [style=black vertex, label={below:$\tilde \alpha$}] (88) at (-7, -5) {};
		\node [style=black vertex, label={left:$\Gamma_0$}] (89) at (-9, -3) {};
	\end{pgfonlayer}
	\begin{pgfonlayer}{edgelayer}
		\draw [style=mydashed] (89)
			 to [in=-90, out=60, looseness=1.50] (35.center)
			 to [in=-90, out=90, looseness=1.50] (31.center);
		\draw [in=-135, out=90, looseness=1.50] (31.center) to (27);
		\draw [in=-135, out=0] (27) to (15.center);
		\draw [in=45, out=-180] (0) to (56.center);
		\draw [in=-45, out=180] (5) to (60.center);
		\draw [style=mydashed] (47.center)
			 to [in=135, out=-45, looseness=1.75] (67.center)
			 to [in=135, out=-45, looseness=1.75] (69.center)
			 to [in=135, out=-45, looseness=1.75] (70.center);
		\draw [style=mydashed] (15.center)
			 to [in=-135, out=45, looseness=1.75] (71.center)
			 to [in=-135, out=45, looseness=1.75] (73.center)
			 to [in=-135, out=45, looseness=1.75] (74.center);
		\draw [in=45, out=-180] (2) to (74.center);
		\draw [style=mydashed] (75.center)
			 to [in=-135, out=45, looseness=1.75] (77.center)
			 to [in=-135, out=45, looseness=1.75] (78.center)
			 to [in=-135, out=45, looseness=1.75] (56.center);
		\draw [style=mydashed] (63.center)
			 to [in=135, out=-45, looseness=1.75] (80.center)
			 to [in=135, out=-45, looseness=1.75] (81.center)
			 to [in=135, out=-45, looseness=1.75] (60.center);
		\draw [in=-45, out=180] (7) to (70.center);
		\draw [in=-135, out=0] (2) to (75.center);
		\draw [in=135, out=0] (7) to (63.center);
		\draw [style=mydashed] (89)
			 to [in=180, out=-30, looseness=1.50] (87.center)
			 to [in=180, out=0, looseness=1.50] (86.center);
		\draw [in=135, out=0, looseness=1.50] (86.center) to (88);
		\draw [style=mydashed] (83.center)
			 to [in=135, out=-45, looseness=1.75] (84.center)
			 to [in=135, out=-45, looseness=1.75] (85.center);
		\draw [in=-45, out=180] (88) to (85.center);
		\draw [in=135, out=0] (82) to (83.center);
		\draw [in=135, out=0] (88) to (47.center);
	\end{pgfonlayer}
\end{tikzpicture}}
  \caption{The situation when $\phi_*(\beta) \in \Gamma_n$.}
  \label{fig:caseB:1}
\end{subfigure}%
\hfill
\begin{subfigure}{.52\textwidth}
  \centering
\resizebox{\textwidth}{!}{\begin{tikzpicture}
	\begin{pgfonlayer}{nodelayer}
		\node [style=smooth n, label={below:$\zeta$}] (2) at (-2, -1) {};
		\node [style=smooth n, label={below:$\phi_*(\zeta)$}] (7) at (-2, -5) {};
		\node [style=none] (15) at (-6, -1) {};
		\node [style=black vertex, label={below:$\alpha$}] (27) at (-7, -1) {};
		\node [style=none] (31) at (-7.75, -1.5) {};
		\node [style=none] (35) at (-8.5, -2.25) {};
		\node [style=none] (47) at (-6, -5) {};
		\node [style=none] (63) at (-1, -5) {};
		\node [style=none] (67) at (-5, -5) {};
		\node [style=none] (69) at (-4, -5) {};
		\node [style=none] (70) at (-3, -5) {};
		\node [style=none] (71) at (-5, -1) {};
		\node [style=none] (73) at (-4, -1) {};
		\node [style=none] (74) at (-3, -1) {};
		\node [style=none] (75) at (-1, -1) {};
		\node [style=none] (77) at (0, -1) {};
		\node [style=none] (78) at (1, -1) {};
		\node [style=none] (80) at (0, -5) {};
		\node [style=none] (81) at (1, -5) {};
		\node [style=none] (100) at (4, -1) {};
		\node [style={G{n}}, label={below:$\beta$}] (101) at (5, -1) {};
		\node [style=none] (102) at (2, -1) {};
		\node [style=none] (103) at (3, -1) {};
		\node [style=none] (104) at (2, -5) {};
		\node [style=none] (82) at (3.25, -4.25) {};
		\node [style=none] (83) at (2.75, -4.25) {};
		\node [style=none] (84) at (2.75, -5.75) {};
		\node [style=none] (85) at (3.25, -5.75) {};
		\node [style=none] (86) at (5.5, -6.5) {};
		\node [style=none] (87) at (5.5, -3.5) {};
		\node [style=none] (88) at (6.5, -5.5) {};
		\node [style=none] (89) at (6.5, -4.5) {};
		\node [style=G-bdry, label={above:$\tilde\beta = \partial D$}] (90) at (3, -5) {};
		\node [style=none] (91) at (3.5, -5) {};
		\node [style=none] (92) at (4, -5) {};
		\node [style=none] (93) at (4.5, -5) {};
		\node [style=none] (94) at (4, -3.5) {};
		\node [style=none] (95) at (4.5, -3.5) {};
		\node [style=none] (96) at (5, -3.5) {};
		\node [style=none] (97) at (4, -6.5) {};
		\node [style=none] (98) at (4.5, -6.5) {};
		\node [style=none] (99) at (5, -6.5) {};
		\node [style=empty, label={below:$\phi_*(\beta)$}] (5) at (5, -5) {};
		\node [style=empty, label={below:$\phi_*(\alpha)$}] (105) at (-11, -5) {};
		\node [style=none] (106) at (-10, -5) {};
		\node [style=none] (107) at (-9, -5) {};
		\node [style=none] (108) at (-8, -5) {};
		\node [style=none] (109) at (-7.5, -4.25) {};
		\node [style=none] (110) at (-8.25, -3.5) {};
		\node [style=black vertex, label={below:$\tilde \alpha$}] (111) at (-7, -5) {};
		\node [style=black vertex, label={left:$\Gamma_0$}] (112) at (-9, -3) {};
	\end{pgfonlayer}
	\begin{pgfonlayer}{edgelayer}
		\draw [style=mydashed] (112)
			 to [in=-90, out=60, looseness=1.50] (35.center)
			 to [in=-90, out=90, looseness=1.50] (31.center);
		\draw [in=-135, out=90, looseness=1.50] (31.center) to (27);
		\draw [in=-135, out=0] (27) to (15.center);
		\draw [style=mydashed] (47.center)
			 to [in=135, out=-45, looseness=1.75] (67.center)
			 to [in=135, out=-45, looseness=1.75] (69.center)
			 to [in=135, out=-45, looseness=1.75] (70.center);
		\draw [style=mydashed] (15.center)
			 to [in=-135, out=45, looseness=1.75] (71.center)
			 to [in=-135, out=45, looseness=1.75] (73.center)
			 to [in=-135, out=45, looseness=1.75] (74.center);
		\draw [in=45, out=-180] (2) to (74.center);
		\draw [style=mydashed] (75.center)
			 to [in=-135, out=45, looseness=1.75] (77.center)
			 to [in=-135, out=45, looseness=1.75] (78.center)
			 to [in=-135, out=45, looseness=1.75] (102.center)
			 to [in=-135, out=45, looseness=1.75] (103.center)
			 to [in=-135, out=45, looseness=1.75] (100.center);
		\draw [style=mydashed] (63.center)
			 to [in=135, out=-45, looseness=1.75] (80.center)
			 to [in=135, out=-45, looseness=1.75] (81.center)
			 to [in=135, out=-45, looseness=1.75] (104.center);
		\draw [in=-45, out=180] (7) to (70.center);
		\draw [in=-135, out=0] (2) to (75.center);
		\draw [in=135, out=0] (7) to (63.center);
		\draw [style=G-domain] (94.center) to (91.center);
		\draw [style=G-domain] (95.center) to (92.center);
		\draw [style=G-domain] (93.center) to (96.center);
		\draw [style=G-domain] (93.center) to (99.center);
		\draw [style=G-domain] (92.center) to (98.center);
		\draw [style=G-domain] (91.center) to (97.center);
		\draw [style=grey edge] (83.center) to (90);
		\draw [style=grey edge] (90) to (85.center);
		\draw [style=grey edge] (82.center) to (90);
		\draw [style=grey edge] (90) to (84.center);
		\draw [style=G-domain] (5) to (87.center);
		\draw [style=G-domain] (5) to (86.center);
		\draw [style=G-domain] (5) to (88.center);
		\draw [style=G-domain] (89.center) to (5);
		\draw [style=G-domain] (90) to (5);
		\draw [in=45, out=-180] (101) to (100.center);
		\draw [in=-45, out=180] (90) to (104.center);
		\draw [style=mydashed] (112)
			 to [in=180, out=-30, looseness=1.50] (110.center)
			 to [in=180, out=0, looseness=1.50] (109.center);
		\draw [in=135, out=0, looseness=1.50] (109.center) to (111);
		\draw [style=mydashed] (106.center)
			 to [in=135, out=-45, looseness=1.75] (107.center)
			 to [in=135, out=-45, looseness=1.75] (108.center);
		\draw [in=-45, out=180] (111) to (108.center);
		\draw [in=135, out=0] (105) to (106.center);
		\draw [in=135, out=0] (111) to (47.center);
	\end{pgfonlayer}
\end{tikzpicture}}
  \caption{The situation when $\phi_*(\beta) \in D \in \mathcal D_n$.}
  \label{fig:caseB:2}
\end{subfigure}%
\caption{End of (Case B) with vertices in $\Gamma_0$, ${\gamcol\Gamma_n}$, $ {\smgamcol\tilde\Gamma_n}$, showing ${\gamdomcol D} \in \mathcal D_n$.}
\label{fig:caseB}
\end{figure}
 
 \textbf{Remarks on the $N$-periodic case.}
For each $1 \le j \le N$ we build a vertex set $\Gamma_{n, (j)}$ and $\tilde \Gamma_{n, (j)}$ as above. We need to start with a (no-)folding tree $T_{(j)} \subset \P^1_{\an, (j)}$ for each $j$ given by applying \autoref{lem:nofolding} to $\phi_*^{(j)} : \P^1_{\an, (j)} \to \P^1_{\an, (j+1)}$. Then we set $\Gamma_{0, (j)} = \Gamma_{(j)} \cup T_{(j)} \cap \GV_{m_0}$. We would also define collections $\mathcal{D}_{n, (j)} \subset \tilde \Gamma_{n, (j)}$ with similar axioms adjusted for periodicity, for instance $\phi_*(D) = \phi_*^{(j)}(D)\subset D' \in \mathcal{D}_{n, (j+1)}$ for every $j$ and $n$. We continue to define the strict rules for adding points to $\Gamma_{n+1, (j)}$, noting the comments in \autoref{rmk:skewstab:extendedfatoujulia}. The challenge again is to prove that if these sets grow forever as $n \to \infty$, we can find a contradiction; in particular we assume that points are added infinitely often to the vertex sets $\Gamma_{n, (j)}$ for a fixed $j$. From here on the proof essentially proceeds as it originally did. 
For instance we can find a $\zeta \in \tilde \Gamma_{n, (j)} \sm \Gamma_{n, (j)}$ such that $\phi_*^{(j)}(\zeta) \in \Gamma_{n+1, (j+1)} \sm \tilde \Gamma_{n, (j+1)}$, and again the problem breaks down into case (A) and (B).
The `no-folding' technique of the proof still works due to our initial choice to exclude the folding trees. Given an interval $(\alpha, \beta)$ not intersecting $T_{(j)}$ with $\Gamma_{0, (j)}$ in a different direction than $\beta$ at $\alpha$, we know that $(\alpha, \beta)$ is mapped injectively by $\phi_*^{(j)}$.
\end{proof}

\section{Counter-examples}\label{sec:skewcounter}

In this section we examine the failure of potential algebraic stability in skew products, showing the hypothesis in \autoref{thm:intro:stabskew} is necessary. We present the following counterexample to illustrate the theory with the simplest coefficients and indices. At the end we provide further discussion about the generality of such examples.

The hypothesis of \autoref{thm:intro:stabskew} is that for any periodic $b \in B$ is not \emph{superattracting}, i.e. $b$ is a \emph{simple} solution to $\phi_1(x) = b$. Therefore, by Riemann-Hurwitz, the base curve $B$ must be rational for any counterexample to exist. The proof of \autoref{thm:intro:stabskew} relies on two key facts about $\phi_*$. First, the (generic) multiplicity of $\phi_*^n(\zeta)$ does not increase under iteration by $\phi_*$. For any superattracting $\phi_*$, this fails for almost every Type II point in $\P^1_\an(\K)$. At the least, this would cause exponentially more points to be added to $\Gamma$ in the smoothing step. Second and more fundamentally, if $\phi_*$ is a simple non-Archimedean skew product then (remarkably) every Type II Julia point of $\phi_*$ is preperiodic. In the superattracting case, the general Julia Type II point has an infinite orbit, so they leave $\Gamma$ and they are always destabilising; see \autoref{lem:skewstab:tail}. \autoref{thm:skewstab:skewcounter} is one of a plethora of skew products with a superattracting fibre and where the associated $\phi_*$ has such a wandering Julia point $\zeta$. However, one can get `lucky', as with the skew product $\phi : X \dashto X$ in \autoref{thm:skewstab:priorcounterex}; there, the Gauss point $\zeta$ happens to be Julia but fixed. The rest of the construction of \autoref{thm:skewstab:skewcounter} relies on the divisor $E$ corresponding to $\zeta$ having an infinite backward orbit on any birationally equivalent surface.

\begin{thm}[\autoref{thm:intro:skewcounter}]\label{thm:skewstab:skewcounter}
 Consider the rational map \[\psi : (x, y) \longmapsto \left((1-x)x^2, (1-x)(x^4y^{-3} + y^3)\right)\] as defined on $\P^1 \times \P^1$. There is no birational map $\psi: X \dashto \P^1 \times \P^1$ conjugating $\psi$ or any of its iterates to an algebraically stable map, even if $X$ is allowed to be singular.
\end{thm}

In \cite{stability} the author gave the following example given to demonstrate that a rational map $\phi$ can be potentially algebraically stable without there existing a stabilisation via birational morphism.

\begin{thm}[{\cite[Theorem 6]{stability}}]\label{thm:skewstab:priorcounterex}
 Let $\phi :\C^2 \dashto \C^2$ be given by \[(x, y) \longmapsto (x^2, x^4y^{-3} + y^3).\]
 Then $\phi$ extends to an algebraically stable rational map $\phi :X \dashto X$ of a Hirzebruch surface $X$. 
 If however $\sigma : (\tilde\phi, \tilde X) \to (\phi, X)$ is the point blowup of $(0, 0) \in X$, then there does not exist a birational morphism $\pi : (\psi, Y) \to (\tilde\phi, \tilde X)$ which stabilises $\tilde \phi$. Furthermore this remains true even if $Y$ is allowed to be singular or if we replace $\phi$ by an iterate $\phi^j$.
\end{thm}

The similarity between the two examples is clear. Either one produces the chaos in the Berkovich projective line $\P^1_\an(\K)$ necessary to prevent algebraic stabilisation over the fibre of $x=0$, except for very special choices like $X$ in \autoref{thm:skewstab:priorcounterex}. The factor $(1-x)$ introduced into the formula for \autoref{thm:skewstab:skewcounter} gives $x=0$ an extra preimage $x=1$, forcing either $\set{x=1}$ or a curve in its backward orbit to fall into a destabilising orbit within $\set{x=0}$.

First we explain \autoref{thm:skewstab:priorcounterex} using the dynamical concepts of non-Archimedean skew products, specifically the Fatou-Julia theory. This will be much more concise and informative than the prior explanation in {\cite{stability}}. Afterwards, we extend the demonstration to prove \autoref{thm:intro:skewcounter}.

\begin{lem}\label{lem:skewstab:fatoudomain}
 Let $\phi_* : \P^1_\an \to \P^1_\an$ be a simple $\k$-rational skew product, and $\Gamma \subset \bH$ be a finite set. Then any F-domain $U \subset \mathcal{F}(\Gamma)$ is contained in the Fatou set of $\phi_*$, $U \subset \F_{\phi, \an}$.
\end{lem}

\begin{proof}
 If $\zeta \in U$ were Julia, then necessarily $\bigcup_n\phi_*^n(U)$ can only omit a finite set of exceptional Type I points; see {\cite[Theorem 3.83]{thesis}}. Therefore it contains $\Gamma$ and so $U$ is a J-domain.
\end{proof}

\begin{lem}\label{lem:skewstab:wanderingjulia}
 Let $\phi : X \dashto X$ be a rational skew product and $\phi_* : \P^1_\an \to \P^1_\an$ the skew product on the Berkovich projective line induced by $\phi$ over some fibre $X_b$ fixed by $\phi$, i.e.\ $\phi_1(b) = b$. Suppose $E \subset X_b$ is a divisor corresponding to a (Type II) Julia point $\zeta \in \J_{\phi, \an}$ which is not preperiodic. Then $E$ is a destabilising curve, so $\phi$ is not algebraically stable.
\end{lem}

\begin{proof}
 Let $\Gamma = \Gamma(X_b) \subset \bH$ be the finite vertex set corresponding to the divisors in $X_b$. Since $\zeta$ has an infinite orbit, some $\phi_*^n(\zeta) = \zeta_1$ must eventually fall outside $\Gamma$ into a $\Gamma$-domain $U$. By \autoref{lem:skewstab:fatoudomain} this is a J-domain since $\zeta_1 \in \J_{\phi, \an}$, and so $\phi_*$ is not analytically stable. Then \autoref{prop:stabskew:analyticallystable} says $\phi$ is not algebraically stable.
\end{proof}

As a consequence, any curve in the backward orbit of this $E$ in the lemma causes destabilising orbits also.

\begin{lem}\label{lem:skewstab:tail}
Let $\phi : X \dashto X$, be a rational skew product over the curve $B$. Suppose $b \in B$ is a $p$-periodic point of $\phi_1$ and $c \in \Orb_{\phi_1}^-(b)$ i.e.\ $\phi_1^N(c) = b$. Let $\left(\phi_*^{(j)} : \P^1_{\an, (j)}\right)_{j=-N}^{p-1}$ be the corresponding preperiodic chain of skew products on the Berkovich projective line, with $\redct_j : \P^1_{\an, (j)} \to Y_{\phi_1^{j+N}(c)}$. Suppose that $\zeta \in \P^1_{\an, (-N)}$ is a Type II point and $E = \redct_{-N}(\zeta) \subset X_c$ is a curve such that $\phi_*^N(\zeta) \in \J_{(0), \an}$ is a Julia point which is not preperiodic. Then $E$ is a destabilising curve for $\phi$.
\end{lem}

\begin{lem}\label{lem:skewstab:intervaljulia}
 Let $\phi$ be the skew product defined in \autoref{thm:skewstab:priorcounterex} or \autoref{thm:skewstab:skewcounter}, and $\phi_* : \P^1_\an \to \P^1_\an$ the skew product on the Berkovich projective line induced by $\phi$ over the fibre $\set{x=0}$. Then $[\zeta(0, \abs x^{\frac 43}), \zeta(0, 1)] \subset \J_{\phi, \an}$, and $\zeta(0, \abs x)$ is not preperiodic.
\end{lem}

\begin{proof}
By {\cite[Theorem 3.37]{thesis}}, $\zeta(0, r) \mapsto \zeta(0, R)$ where $R = \max (\abs{x}^2r^{-3/2}, r^{3/2})$. Hence, on the interval $(0, \infty) \subset \P^1_\an$, the dynamics is described by $\zeta(0, \abs x^t) \mapsto \zeta(0, \abs x^{T_\phi(t)})$, where
\[T_\phi : t\longmapsto 
\begin{cases}
 \frac {3}{2}t & t \le \frac 23\\
 2 - \frac {3}{2}t & t > \frac 23
\end{cases}
\]
\emph{Claim}: For any non-trivial subinterval $I \subseteq (0, \frac 43)$ there is an $N$ with $T_\phi^n(I) \supseteq [\half, 1]$. 
 \thesisarticle{We proved this in the \hyperref[proof:stab:counterex]{proof~of~\autoref*{thm:intro:counterex}}.
 }{
Note first that $T_\phi((1, \frac 43)) = (0, \half)$ and the points in $(0, \half)$ are repelled away from $0$ into $[\half, 1]$, which is forward invariant. Therefore, it is enough to prove the claim with $I \subseteq [\half, 1]$. 
The map $T_\phi(t)$ expands, by a factor of $\frac 32$, the lengths of any subinterval which does not include $t = \frac 23$. If an interval $(a, b)$ does include $\frac 23$, then this applies to $(a, \frac 23)$ or $(\frac 23, b)$, so $T_\phi((a, b))$ has length at least $\max\{\frac 32(b-\frac 23), \frac 32(\frac 23-a)\} \ge \frac 32 \cdot \half (b-a) = \frac 34(b-a)$. So, if no two consecutive intervals in the sequence $T_\phi^n(I)$ contain $\frac 23$, then for each $n$ the length of $T_\phi^{n+2}(I)$ is $\frac 32 \cdot \frac 34 = \frac 98 > 1$ times greater than the length of $T_\phi^n(I)$. Since $I \subset [0, 1]$ and $T_\phi([0, 1]) = [0, 1]$, this cannot occur indefinitely. Thus, for some $n$, both $T_\phi^n(I)$ and $T_\phi^{n+1}(I)$ contain $\frac 23$. Under $T_\phi$ we have
$\frac 23 \mapsto 1 \mapsto \frac 12$, therefore $[\half, 1] \subset T_\phi^{n+2}(I)$.
 }
 
 On $[\half, 1]$ one can check that $t= \frac 45$ is a fixed point of $T_\phi$, so $\zeta = \zeta(0, \abs x^{4/5})$ is a fixed point of $\phi_*$. Further, this Type II point has \dmult{} $\frac 32$, so it is \numericallyrepelling{} and hence Julia by {\cite[Theorem 3.72]{thesis}}. By the claim, any subinterval of $(0, \frac 43)$ intersects $\Orb_{T_\phi}^-(\frac 45)$ and so $\Orb_{\phi_*}^-(\zeta)$ is dense in $(\zeta(0, 1), \zeta(0, \abs x^{\frac 43}))$. Since the Julia set is closed and backward invariant we get $[\zeta(0, 1), \zeta(0, \abs x^{\frac 43})] \subset \J_{\phi, \an}$.
 
 Now we show that $\zeta(0, \abs x^1)$ is not preperiodic under $\phi_*$. It is enough to show that under $T_\phi$, the parameter $1$ has an infinite orbit.
\[1 \mapsto \frac 12 \mapsto \frac 34 \mapsto \frac 78 \mapsto \frac{11}{16} \mapsto \cdots\]
To justify this, suppose that $a$ odd and compute $T_\phi\left(\frac a{2^n}\right)$.
\[\frac a{2^n} \longmapsto
\begin{cases}
 \frac {3a}{2^{n+1}} & \mathmakebox[4mm]{\frac a{2^n}} < \frac 23\\
 2 - \frac {3a}{2^{n+1}} = \frac{2^{n+2} - 3a}{2^{n+1}} & \mathmakebox[4mm]{\frac ab} > \frac 23
\end{cases}
\]
Since both $3a$ and $2^{n+2} -3a$ are odd, the image has a larger denominator. By induction, the orbit is infinite.
\end{proof}

\begin{proof}[Proof of \autoref{thm:skewstab:priorcounterex}]
  The initial stabilisation $\rho : (\tilde \phi, \P^1 \times \P^1) \dashto (\phi, X)$ is a blowup and blowdown of the line at infinity. Blowing up the origin to produce $\tilde X$ produces an exceptional divisor $E_1$ which is the reduction of $\zeta(0, \abs x)$ to $\tilde X$. Instead of proving the theorem as stated, we will prove something stronger:  
  For any rational surface $Y$ and birational map $\rho : Y \dashto \tilde X$, such that $Y$ contains (a curve which is the proper transform) $E_1$, $Y$ must contain a destabilising curve. Specifically, this will be a divisor above $x=0$, corresponding to a Type II Julia point $\zeta$ which is not preperiodic.
 
Suppose $\redct_{Y, 0}(\zeta(0, \abs x)) = E_1 \subset Y$. We know from \autoref{lem:skewstab:intervaljulia} that $\zeta(0, \abs x)$ is a Julia point and not preperiodic. Now, by \autoref{lem:skewstab:wanderingjulia}, $E_1$ has to be a destabilising curve. We remark that these assertions did not require smoothness or any property of $Y$ or $\phi$ except that some curve reduces to $\zeta(0, \abs x)$. This holds if we replace $\phi$ with an iterate $\phi^j$, since $\J_{\phi, \an} = \J_{\phi^j, \an}$.
\end{proof}

The proof of \autoref{thm:skewstab:skewcounter} now follows with little more work because the formulae are so similar. Here is a concise account of the proof below. We see in formula, \[\psi : (x, y) \longmapsto \left((1-x)x^2, (1-x)(x^4y^{-3} + y^3)\right), \]
  that $1 \in \psi_1^{-1}(0)$, and furthermore $\set{x=1}$ has an infinite backward orbit of fibres in $\P^1\times \P^1$. If we blow up $\P^1\times \P^1$ at the origin, producing an exceptional curve $E_1$, one can check that $\psi(\set{x=1}) = E_1$. For any $n \ge 1$ and $c\in \k^\times$ such that $\psi_1^n(c) = 0$, hence $\psi^n(\set{x=c}) = E_1$. The corresponding picture in the Berkovich projective line is $\psi_*^n(\zeta(0, 1)) = \zeta(0, \abs x)$. Since a birational transformation can only collapse or modify finitely many fibres, this remains true for infinitely many $n$ on any surface. By \autoref{lem:skewstab:intervaljulia} $\zeta(0, \abs x)$ is Julia, and not preperiodic. Therefore on an arbitrary birational model, for infinitely many $n$ and $c \in \phi_1^{-n}(0)$, $\set{x=c}$ is a destabilising curve by \autoref{lem:skewstab:tail}. 

\begin{proof}[Proof of \autoref{thm:skewstab:skewcounter}]
 Consider initially $\psi$ as a rational map $\P^1\times \P^1$. Over any $b \in \P^1$, the Gauss point, $\redct_b(\zeta(0, 1)) = F_b$ reduces to the (entire) fibre of $\P^1\times \P^1$ over $b$. Clearly, $\Orb_{\psi_1}^-(1)$ is infinite since $\psi_1(1) = 0$ is fixed but not totally ramified by $\psi_1$. Let $\rho : Y \dashto \P^1\times \P^1$ be an arbitrary birational modification. The inverse $\rho^{-1}$ can contract at most finitely many curves, in particular only finitely many fibres above $\Orb_{\psi_1}^-(1)$. Therefore we can find an $N \in \N$ and $c \in \psi_1^{-N}(1)$ such that the proper transform of $F_b$ by $\rho^{-1}$ is the fibre $Y_c \cong \P^1$ above $c$, and thus $\redct_{Y_c}(\zeta(0, 1)) = Y_c$. Consider the (preperiodic) chain of skew products $\left(\psi_*^{(j)} : \P^1_{\an, (j)}\right)_{j=0}^{N+1}$ with $\redct_j : \P^1_{\an, (j)} \to Y_{\psi_1^j(c)}$. 

Consider the map $\psi_* : \P^1_{\an, (N)} \to \P^1_{\an, (N+1)}$ corresponding to $\psi_1 : 1 \mapsto 0$. We have $\P^1_{\an, (N+1)}$ defined over the Puiseux series in $x$ with norm $\abs\cdot_x$, whilst $\P^1_{\an, (N)}$ defined over the Puiseux series in $x' = 1-x$ with norm $\abs\cdot_{x'}$. Then $\psi_*$ is a simple skew product with respect to $x$ and $x'$, because $\abs{\psi_1^*(x)}_{x'} = \abs{x'(1-x')^2}_{x'} = \abs{x'}_{x'}^1$. The expression $(1-x)(x^4y^{-3} + y^3)$ can be rewritten $x'((1-x')^4y^{-3} + y^3)$, so one can check that $\psi_{2*}(\zeta(0, 1)) = \zeta(0, \abs {x'})$. Now, $\psi_{1*}$ fixes every point on $[0, \infty]$ and has scale factor $1$, so after applying $\psi_{1*}$ we obtain $\psi_*(\zeta(0, 1)) = \zeta(0, \abs x)$. One can check that the critical points of $\psi_1$ are $\set{0, \frac 23, \infty}$ and that that the iterates $(\psi_1^j(\frac23))_{j=0}^\infty$ wander and do not include $1$ or $0$. Hence $\psi_1$ is unramified at $\psi_1^j(c)$ for every $j \le N$, so similarly $\psi_*^{(j)} : \P^1_{\an, (j)} \to \P^1_{\an, (j+1)}$ is a simple skew product. Furthermore, for every $j < N$, this skew product has good reduction, whence $\psi_*^{(j)}(\zeta(0, 1)) = \zeta(0, 1)$. 
 In summary, $\psi_*^{N+1}(\zeta(0, 1)) = \zeta(0, \abs x) \in \P^1_{\an, (N+1)}$. 
 
 The action of $\psi_*$ on $[\zeta(0, 1), \zeta(0, \abs x)]$ is the same as for the skew product $(x, y) \mapsto (x^2, x^4y^{-3} + y^3)$ studied in \autoref{thm:skewstab:priorcounterex} essentially because $(1-x)$ is a unit of the Puiseux series in $x$.
 Using the same proof as above for (the advanced version of) \autoref{thm:skewstab:priorcounterex}, we conclude that $\zeta(0, \abs x) \in \J_{(N+1), \an}$ is Julia and not preperiodic. 
 Finally, by \autoref{lem:skewstab:tail}, $\psi$ is not algebraically stable since $\redct_0(\zeta(0, 1)) = Y_b$ and $\psi_*^{N+1}(\zeta(0, 1)) = \zeta(0, \abs x) \in \J_{\psi, \an}$ is a wandering Julia point.
\end{proof}

Through \autoref{thm:skewstab:skewcounter} we can see a general strategy to produce skew products on $\P^1\times\P^1$ which are not potentially algebraically stable. First, write down $\phi_1 : \P^1 \to \P^1$ with a superattracting point, say $0$, which is not exceptional (finite backward orbit). Specifically, suppose $b \in \phi_1^{-1}(0)$ is not in the cycle of $0$. Second, write down a $\phi_2(x, y) \in \k(x, y)$ which has bad reduction at $x=0$ and the following properties.
\begin{enumerate}
 \item The non-Archimedean skew product $\phi_*$ induced by $(\phi_1(x), \phi_2(x, y))$ above $x=0$ has a Type II point $\zeta$ of generic multiplicity $1$ which is Julia and not preperiodic.
 \item $\phi_2$ has good reduction for every $c \in \Orb_{\phi_1}^-(b) \sm \set b$.
 \item On the skew product $\phi_* : \P^1_{\an, (x-b)} \to \P^1_{\an, (x)}$ corresponding to $\phi_1 : b \mapsto 0$, we have $\phi_*(\zeta(0, 1)) = \zeta$. It follows from good reduction that for any \sloppy\mbox{$c \in \phi_1^{-N}(b)$}, we have $\phi_*^N(\zeta(0, 1)) = \zeta(0, 1)$ for the skew product $\phi_*^N$ corresponding to $\phi_1^N : c \mapsto b$.
\end{enumerate}
For the first part, one can write down an expanding (piecewise linear) interval map (such as $T_\phi$ in the counterexamples) and realise it with some $\phi_*$ on a forward invariant interval of $\P^1_\an$. For the third part, it may help to find an element $\theta \in \PGL(2, \k(x-b))$ which has good reduction on every fibre except $x=b$, and compose $\psi = \theta \circ \phi$. Moving from example \autoref{thm:skewstab:priorcounterex} to \autoref{thm:skewstab:skewcounter} we chose the transformation $\theta(x, y) = (1-x)y$, which does not disturb the chaotic map on the interval $(0, \infty) \subset \P^1_\an$.

 \bibliographystyle{alpha}
\bibliography{Thesis}

\end{document}